\documentclass[11pt]{article}
\usepackage[utf8]{inputenc}

\title{One-step ahead  sequential Super  Learning from  short times  series of
  many  slightly  dependent  data,  and   anticipating  the  cost  of  natural
  disasters}

\author{Geoffrey    Ecoto$^{1,2}$,    Aurélien   F.     Bibaut$^3$,    Antoine
  Chambaz$^2$\\~\\ 
  $^1$ Caisse Centrale de Réassurance\\
  $^2$ MAP5 (UMR CNRS 8145), Université de Paris\\
  $^3$ Netflix}

\date{\today}
\usepackage[dvipsnames]{xcolor}
\usepackage[colorlinks=true,allcolors=blue,hyperfootnotes=false]{hyperref}
\usepackage[margin=1in]{geometry}
\usepackage[mathscr]{euscript}
\usepackage{setspace}
\usepackage{amssymb,amsmath,amsthm,stmaryrd}
\usepackage{graphicx}
\newtheorem{theorem}{Theorem}
\newtheorem{corollary}[theorem]{Corollary}
\newtheorem{assumption}{Assumption}
\newtheorem{lemma}[theorem]{Lemma}

\usepackage{natbib}
\allowdisplaybreaks
\newcommand{\1}{\textbf{1}}
\newcommand{\argmin}{\mathop{\arg\min}}
\newcommand{\bbR}{\mathbb{R}}
\newcommand{\calA}{\mathcal{A}}

\newcommand{\calF}{\mathcal{F}}
\newcommand{\calG}{\mathcal{G}}
\newcommand{\calI}{\mathcal{I}}
\newcommand{\calO}{\mathcal{O}}
\newcommand{\calX}{\mathcal{X}}
\newcommand{\calZ}{\mathcal{Z}}
\newcommand{\Exp}{\mathbb{E}}
\newcommand{\LL}{\mathscr{L}}
\newcommand{\Prob}{\mathbb{P}}
\newcommand{\RR}{\mathscr{R}}
\newcommand{\tZeta}{\tilde{\Zeta}}
\newcommand{\var}{\mathrm{var}}
\newcommand{\Var}{\mathbb{V}\mathrm{ar}}
\DeclareMathOperator{\Summ}{Summ}
\newcommand{\Zeta}{\mathscr{Z}}
\definecolor{ballblue}{rgb}{0.13, 0.67, 0.8}

\begin{document}

\maketitle

\begin{abstract}
  Suppose that  we observe  a short time  series where  each time-$t$-specific
  data-structure consists of  many slightly dependent data indexed  by $a$ and
  that we want to estimate a feature of the law of the experiment that depends
  neither on  $t$ nor  on $a$.   We develop  and study  an algorithm  to learn
  sequentially which base algorithm in a user-supplied collection best carries
  out the estimation  task in terms of excess risk  and oracular inequalities.
  The analysis, which uses dependency graph to model the amount of conditional
  independence  within each  $t$-specific data-structure  and a  concentration
  inequality by~\citet{Janson04},  leverages a  large ratio  of the  number of
  distinct $a$-s to the degree of the  dependency graph in the face of a small
  number of $t$-specific data-structures.   The so-called one-step ahead Super
  Learner  is applied  to the  motivating example  where the  challenge is  to
  anticipate the cost of natural disasters in France.
\end{abstract}

\section{Introduction}
\label{sec:intro}

\paragraph*{Caisse Centrale de Réassurance  and the cost of natural disasters in France.}

In  France,  Law  n\textsuperscript{o}82-600  of  July  13th  1982  imposes  a
compulsory extension of the guarantee for all property insurance contracts for
the coverage of natural catastrophes. This  law defines the legal framework of
the  natural  disasters  compensation  scheme, of  which  Caisse  Centrale  de
Réassurance  (CCR)  is  a  major  actor in  France.   With  the  French  State
guarantee,  CCR  provides its  cedents\footnote{A  cedent  is  a party  in  an
  insurance  contract  that  passes   the  financial  obligation  for  certain
  potential losses to the insurer. In  return for bearing a particular risk of
  loss,  the  cedent   pays  an  insurance  premium.}    operating  in  France
(\textit{i.e.},  the   insurance  companies  operating  in   France  that  CCR
reinsures) with unlimited coverage against  natural catastrophes.  In order to
better  anticipate  the risks,  CCR  has  developed  an expertise  in  natural
disasters  modeling.  The  so-called  ``cat models"~\citep{CatModels}  exploit
portfolios and  claims data collected  from CCR's  cedents to enable  a better
appreciation  of the  exposures\footnote{The state  of being  subject to  loss
  because of some  hazard or contingency.} of  CCR, of its cedents  and of the
French  State.   Our  study  proposes  a new  method  to  better  predict  the
aforementioned exposures. Termed ``one-step ahead sequential Super Learning'',
rooted  in statistical  theory, the  method allows  to learn  from short  time
series of many slightly dependent data.

\paragraph*{Statistical challenges.}

Developing  such  a method  presents  several  technical challenges.   From  a
theoretical   point  of   view,  we   have  to   deal  with   a  time   series
$(\bar{O}_{t})_{t\geq  1}$  whose  time-$t$-specific  component  $\bar{O}_{t}$
consists of  a large  collection $(O_{\alpha,t})_{\alpha  \in \calA}$  of data
that are dependent but such that there is a large amount of independence among
them.  The time series  is observed only at a limited number  of time steps, a
drawback   that   could   be   mitigated   by   the   large   cardinality   of
$\calA$. Furthermore, for reasons that we  will present later on, we favor the
development of  a learning  algorithm that  works in  an online  fashion.  The
learning algorithm should build upon a library of competing algorithms, either
to select  the one  that performs  best or  to combine  the algorithms  into a
single  meta-algorithm   that  performs  almost   as  well  as   all  possible
combinations  thereof (this  is known  as stacking,  or aggregating,  or Super
Learning in the literature). Of course, assessing the said performances is not
easy,  notably  because  it  requires some  form  of  online  cross-validation
procedure.  From the  applied point of view, assembling the  learning data set
is  difficult because  the data  come from  many sources  and take  on various
shapes.  Moreover, some of the data are only partially available. Details will
be given later on.


\paragraph{Organization of the article.}

Section~\ref{sec:theory} presents the theoretical  development and analysis of
the one-step ahead sequential Super  Learner.  Readers who are more interested
in    the    application    than    in    the    theory    could    jump    to
Section~\ref{subsec:summary}  for  a summary.   Section~\ref{sec:anticipating}
presents the  complete application.   The main objective  is exposed  in finer
detail; the actual implementation of  the algorithm is described; the obtained
results are reported and  commented upon.  Section~\ref{sec:discussion} closes
the article on a discussion. Further details are given in the appendix.

\section{A new result for the one-step ahead sequential Super Learner}
\label{sec:theory}

Let $(\bar{O}_{t})_{t\geq  1}$ be a time-$t$-ordered  sequence of observations
where    each    $\bar{O}_{t}$    is    in   fact    a    finite    collection
$(O_{\alpha, t})_{\alpha  \in \calA}$  of $(\alpha,t)$-specific elements  of a
measured space $\calO$.  We are  especially interested in situations where the
variables  $(O_{\alpha,t})_{\alpha  \in  \calA}$ are  conditionally  dependent
given                            the                            $\sigma$-field
$F_{t-1}  := \sigma(O_{\alpha,\tau}  : \alpha  \in \calA,  1 \leq  \tau <  t)$
generated  by past  observations (by  convention, $F_{0}  := \emptyset$),  but
there is a large amount of conditional independence between them.

We rely  on conditional dependency graphs  to model the amount  of conditional
independence.\footnote{\citet{Janson04}   exploits   the   finer   notion   of
  fractional chromatic numbers.}

\begin{assumption}
  \label{assum:A0}
  There  exists  a  graph  $\calG$  with  vertex  set  $\calA$  such  that  if
  $\alpha  \in  \calA$  is  not  connected  by  any  edge  to  any  vertex  in
  $\calA' \subset \calA$, then  $O_{\alpha,t}$ is conditionally independent of
  $(O_{\alpha',t})_{\alpha'  \in \calA'}$  given  $F_{t-1}$  and (possibly)  a
  known,  fixed summary  measure  $\bar{Z}_{t}:=  \Summ(\bar{O}_{t})$ of  each
  observation  $\bar{O}_{t}$.\footnote{This notion  of conditional  dependency
    graph    is     weaker    than     the    one    that     requires    that
    $(O_{\alpha,t})_{\alpha           \in            \calA_{1}}$           and
    $(O_{\alpha,t})_{\alpha \in \calA_{2}}$ be conditionally independent given
    $F_{t-1}$ and  $\bar{Z}_{t}$ whenever $\calA_{1}, \calA_{2}$  are disjoint
    subsets of $\calA$ with no edge between them.}
\end{assumption}
For   every  $t\geq   1$   the  summary   measure   $\bar{Z}_{t}$  writes   as
$\bar{Z}_{t} :=  (Z_{\alpha,t})_{\alpha \in \calA} \in  \calZ^{\calA}$.  It is
said \textit{fixed} because it is  derived from $\bar{O}_{t}$ by evaluating at
$\bar{O}_{t}$  the  fixed (in  $t\geq  1$  and  $\alpha \in  \calA$)  function
$\Summ$. The  adverb \textit{possibly}  hints at the  case where  $\Summ$ maps
every $\bar{O}_{t}$ to an uninformative, empty summary.

We  let   $\deg(\calG)$  denote   1  plus  the   maximum  degree   of  $\calG$
(\textit{i.e.}, 1  plus the  largest number  of edges that  are incident  to a
vertex  in  $\calG$).  The  smaller  is  $\deg(\calG)$, the  more  conditional
independence we can rely on.

Our  main objective  is  to estimate  a feature  $\theta^{\star}$  of the  law
$\Prob$  of  $(\bar{O}_{t})_{t \geq  1}$,  an  element  of a  parameter  space
$\Theta$ that is  known to minimize over  $\Theta$ the risk induced  by a loss
$\ell$  and $\Prob$.   We consider  the specific  situation where  the feature
$\theta^{\star}$ can also be defined as  the shared minimizer over $\Theta$ of
all the risks induced  by a loss $\ell$ and all  the conditional marginal laws
of   $O_{\alpha,t}$    given   $Z_{\alpha,t}$    (``all''   refers    to   all
$\alpha \in \calA$ and $t \geq 1$).

For instance, we  can address a situation where,  firstly, each $O_{\alpha,t}$
decomposes   as   $O_{\alpha,t}    :=   (X_{\alpha,t},   Y_{\alpha,t})$   with
$X_{\alpha,t} \in \calX$ a  collection of $(\alpha,t)$-specific covariates and
$Y_{\alpha,t} \in [-1,1]$ a corresponding outcome of interest; secondly, under
$\Prob$, the exists a (fixed) graph $\calG$ with vertex set $\calA$ such that,
for all $t\geq 1$,  if $\alpha \in \calA$ is not connected by  any edge to any
vertex  in  $\calA'  \subset  \calA$,  then  $O_{\alpha,t}$  is  conditionally
independent  of   $(O_{\alpha',t})_{\alpha'  \in  \calA'}$   given  $F_{t-1}$;
thirdly,       there      exists       under      $\Prob$       (a      fixed)
$\theta^{\star}:\calX          \to          [-1,1]$         such          that
$\Exp(Y_{\alpha,t}|X_{\alpha,t} =  x, F_{t-1})  = \theta^{\star} (x)$  for all
$x \in  \calX$.  In that  situation, the loss  $\ell$ can be  the least-square
loss  function  that  maps  any  $\theta:\calX \to  [-1,1]$  to  the  function
$(x,y) \mapsto (y - \theta(x))^{2}$.   Note that here, every $Z_{\alpha,t}$ is
empty.

Generally, we make the following assumption.
\begin{assumption}
  \label{assum:A1}
  There exists a  loss function $\ell : \Theta \to  \bbR^{\calO \times \calZ}$
  such that the  feature of interest $\theta^{\star}$ minimizes  all the risks
  $\theta    \mapsto     \Exp[\ell(\theta)(O_{\alpha,t},    Z_{\alpha,t})    |
  Z_{\alpha,t},   F_{t-1}]$   over   $\Theta$,  ``all''   referring   to   all
  $\alpha \in \calA$  and $t \geq 1$.  Moreover, for  every $\theta\in \Theta$
  and sequence  $(\theta_{t})_{t \geq 1}$  of elements of $\Theta$  adapted to
  $(F_{t})_{t  \geq  1}$  (\textit{i.e.},   such  that  each  $\theta_{t}$  is
  $F_{t}$-measurable),  for   all  $t   \geq  2$  and   non-negative  integers
  $\varepsilon_{1},            \varepsilon_{2}$           such            that
  $\varepsilon_{1}+ \varepsilon_{2} = 2$,
  \begin{multline*}
    \Exp\left[\sum_{\alpha   \in    \calA}   (\ell(\theta_{t-1})(O_{\alpha,t},
      Z_{\alpha,t}))^{\varepsilon_{1}}    \times   (\ell(\theta)(O_{\alpha,t},
      Z_{\alpha,t}))^{\varepsilon_{2}}\middle|
      \bar{Z}_{t}, F_{t-1}\right] \\
    =  \sum_{\alpha   \in  \calA}  \Exp\left[(\ell(\theta_{t-1})(O_{\alpha,t},
      Z_{\alpha,t}))^{\varepsilon_{1}}    \times   (\ell(\theta)(O_{\alpha,t},
      Z_{\alpha,t}))^{\varepsilon_{2}} \middle| Z_{\alpha,t}, F_{t-1}\right].
  \end{multline*}
\end{assumption}
Assumption \textbf{A\ref{assum:A1}}  guarantees some  form of  stationarity in
$\Prob$ pertaining to its feature  of interest $\theta^{\star}$.  Thanks to it
there    is    hope    that     we    can    learn    $\theta^{\star}$    from
$\bar{O}_{1},  \ldots, \bar{O}_{t}$  even with  $t$ small  if the  cardinality
$|\calA|$ of $\calA$ is large (in  fact, if the ratio $|\calA|/\deg(\calG)$ is
large).

Section~\ref{subsec:SL} presents the one-step  ahead sequential Super Learner,
a  collection  of   assumptions  on  the  law  $\Prob$  of   the  time  series
$(\bar{O}_{t})_{t\geq  1}$  and  on  its  feature  $\theta^{\star}$,  and  our
theoretical  analysis  of  the   one-step  ahead  sequential  Super  Learner's
performance under these  assumptions.  Section~\ref{subsec:summary} summarizes
the  content of  Section~\ref{subsec:SL} and  Section~\ref{subsec:SL:comments}
gathers  comments on  Section~\ref{subsec:SL}.   The proofs  are presented  in
Appendix~\ref{app:strong:convexity} and \ref{app:proofs}.

\subsection{The  one-step  ahead sequential  Super  Learner  and its  oracular
  performances}
\label{subsec:SL}

\paragraph*{The one-step ahead sequential Super Learner.}

Let $\widehat{\theta}_{1}, \ldots, \widehat{\theta}_{J}$  be $J$ algorithms to
learn  $\theta^{\star}$ from  $(\bar{O}_{t})_{t\geq 1}$.   In words,  for each
$j \in  \llbracket J\rrbracket:=\{1,  \ldots, J\}$, $\widehat{\theta}_j$  is a
procedure that, for every $t \geq 1$, maps $\bar{O}_{1}, \ldots, \bar{O}_t$ to
an  element   of  a  $j$-specific   subset  $\Theta_j$  of   $\Theta$,  namely
$\theta_{j,t}  \in  \Theta_{j}$ (by  convention,  $\theta_{j,0}$  is a  fixed,
pre-specified element  of $\Theta_{j}$).  The one-step  ahead sequential Super
Learner that  we are about  to introduce is  a meta-algorithm that  learns, as
data accrue, which algorithm in the aforementioned collection performs best.

Strictly speaking,  the one-step  ahead sequential Super  Learner really  is an
online algorithm if each of the $J$ algorithms is online, that is, if for each
$j\in\llbracket  J\rrbracket$  and $t\geq  1$,  the  making of  $\theta_{j,t}$
consists  in  an  update  of  $\theta_{j,t-1}$ based  on  newly  accrued  data
$\bar{O}_{t}$. If  that is not  the case, then the  Super Learner is  merely a
sequential algorithm, updated at every time step $t$.

The  measure of  performance takes  the form  of an  average cumulative  risk
 conditioned    on    the    observed    sequence    $(\bar{Z}_{t})_{t    \geq
   1}$.          For          every         $j          \in         \llbracket
 J\rrbracket$, the  risk (for  short) of  $\widehat{\theta}_{j}$ till  time $t
 \geq 1$ is defined as
\begin{align}
  \label{eq:risk}
  \widetilde{R}_{j,t}
  &               :=               \frac{1}{t}               \sum_{\tau=1}^{t}
    \Exp\left[\bar{\ell}(\theta_{j,\tau-1})(\bar{O}_{\tau},
    \bar{Z}_{\tau})    \middle|   \bar{Z}_{\tau},    F_{\tau-1}\right]   \quad
    \text{where} \\
  \label{eq:bar:ell}
  \bar{\ell}(\theta)(\bar{O}_{\tau}, \bar{Z}_{\tau})
  &       :=       \frac{1}{|\calA|}       \sum_{\alpha       \in       \calA}
    \ell(\theta)(O_{\alpha,\tau},  Z_{\alpha,\tau})  \quad   \text{for  all  }
    \theta \in \Theta, \tau \geq 1.
\end{align}
The empirical counterpart of \eqref{eq:risk} is
\begin{equation}
  \label{eq:erisk}
  \widehat{R}_{j,t}
  := \frac{1}{t} \sum_{\tau=1}^{t} \bar{\ell}(\theta_{j,\tau-1})(\bar{O}_\tau,
  \bar{Z}_{\tau})
  =   \frac{1}{t|\calA|}   \sum_{\tau=1}^{t}   \sum_{\alpha   \in   \calA}
  \ell(\theta_{j, \tau-1})(O_{\alpha,\tau}, Z_{\alpha,\tau}). 
\end{equation}
At each time $t \geq 1$, the collection of $(j,t)$-specific empirical risks is
minimized at index $\widehat{j}_{t}$:
\begin{equation}
  \label{eq:oSL}
  \widehat{j}_{t} \in \argmin_{j \in \llbracket J\rrbracket} \widehat{R}_{j,t}
\end{equation}
(the unlikely  ties are  broken arbitrarily).   The one-step  ahead sequential
Super Learner  is the meta-algorithm  that learns $\theta^{\star}$  by mapping
$\bar{O}_{1},  \ldots, \bar{O}_{t}$  to  $\theta_{\widehat{j}_t,t}$ for  every
$t \geq 1$.

To assess  how well the one-step  ahead sequential Super Learner  performs, we
compare   its  risk   to  that   of   the  oracular   algorithm  that   learns
$\theta^{\star}$   by   mapping    $\bar{O}_{1},   \ldots,   \bar{O}_{t}$   to
$\theta_{\widetilde{j}_t,t}$ at each time $t \geq 1$, where
\begin{equation}
  \label{eq:oracle}
  \widetilde{j}_t \in \argmin_{j\in \llbracket J\rrbracket} \widetilde{R}_{j,t}
\end{equation}
(again, the unlikely ties are broken arbitrarily). This is discussed next.

\paragraph*{Comparing  the  one-step ahead  sequential  Super  Learner to  its
  oracular counterpart.}

So       far       we       have        defined       the       risks       of
$\widehat{\theta}_{1}, \ldots, \widehat{\theta}_{J}$, see \eqref{eq:risk}.  By
analogy,  for every  $\theta \in  \Theta$  and $t  \geq  1$, let  the risk  of
$\theta$ at time $t$ be
\begin{equation*}
  \widetilde{R}_{t}(\theta) := \frac{1}{t}               \sum_{\tau=1}^{t}
  \Exp\left[\bar{\ell}(\theta)(\bar{O}_{\tau},      \bar{Z}_{\tau})      \middle|
    \bar{Z}_{\tau}, F_{\tau-1}\right]. 
\end{equation*}
The risk $\widetilde{R}_{t}(\theta)$ can be  interpreted as the risk till time
$t\geq    1$     of    a     dummy    algorithm    that     constantly    maps
$\bar{O}_{1}, \ldots,  \bar{O}_{t}$ to $\theta$  (the algorithm is  said dummy
because it does not learn).  Let $\theta^{\circ} \in \Theta$ be such that
\begin{equation*}
  \widetilde{R}_{t}(\theta^{\circ}) \leq \min_{j \in \llbracket J\rrbracket}  \min_{\theta \in
    \Theta_{j}} \widetilde{R}_{t}(\theta).
\end{equation*}
Under   \textbf{A\ref{assum:A1}},    $\theta^{\circ}$   could   be    set   to
$\theta^{\star}$, but  other choices might  be made on  a case by  case basis.
Our main  results compare the  excess risks  of the one-step  ahead sequential
Super Learner and of the oracle, that is, they compare
\begin{equation*}
  \widetilde{R}_{\widehat{j}_{t},t}  - \widetilde{R}_{t}(\theta^{\circ})  \quad \text{to}
  \quad \widetilde{R}_{\widetilde{j}_{t},t} - \widetilde{R}_{t}(\theta^{\circ}).  
\end{equation*}
They rely on the following assumptions.

For          every          $\theta         \in          \Theta$,          let
$\Delta^{\circ}\ell(\theta) := \ell(\theta) - \ell(\theta^{\circ})$.

\begin{assumption}
  \label{assum:A2}
  There        exists        $b_{1}        >        0$        such        that
  $\sup_{\theta   \in   \Theta}   \|\Delta^{\circ}\ell(\theta)\|_\infty   \leq
  b_{1}$.  Moreover there  exists $b_{2}  \in ]0,  2b_{1}]$ such  that, almost
  surely, for all $\alpha \in \calA$, $t \geq 1$ and $\theta \in \Theta$,
  \begin{equation*}
    \left|\Delta^{\circ}\ell(\theta) (O_{\alpha,t},  Z_{\alpha,t}) -  \Exp \left[
        \Delta^{\circ}\ell(\theta)(O_{\alpha,t},     Z_{\alpha,t})    \middle|
        Z_{\alpha,t}, F_{t-1} \right]\right| \leq b_{2}. 
  \end{equation*}
\end{assumption}

\begin{assumption}
  \label{assum:A3}
  There exist $\beta \in ]0,1]$ and $\gamma > 0$ such that, almost surely, for
  all $\alpha \in \calA$, $t \geq 1$ and $\theta \in \Theta$,
  \begin{equation*}
    \Exp           \left[           \Big(\Delta^{\circ}\ell(\theta)(O_{\alpha,t},
      Z_{\alpha,t})\Big)^{2}  \middle|Z_{\alpha,t},  F_{t-1}  \right]  \leq  \gamma
    \Big(\Exp   \left[   \Delta^{\circ}\ell(\theta)(O_{\alpha,t},   Z_{\alpha,t})
      \middle| Z_{\alpha,t}, F_{t-1} \right]\Big)^{\beta}.
\end{equation*}
\end{assumption}

\begin{assumption}
  \label{assum:A4}
  There   exists   $v_{1}  >   0$   such   that,   almost  surely,   for   all
  $\alpha \in \calA$, $t\geq 1$ and $\theta \in \Theta$,
  \begin{equation*}
    \Var\left[\Delta^{\circ}\ell(\theta)(O_{\alpha,t},
      Z_{\alpha,t}) \middle| Z_{\alpha,t}, F_{t-1}\right] \leq v_{1}.
  \end{equation*}
\end{assumption}

Assumption  \textbf{A\ref{assum:A3}}  is  a so-called  ``variance  bound'',  a
well-known concept  in statistical  learning theory~\citep{BBM05,K06,BJMcA06}.
Under \textbf{A\ref{assum:A2}}, the radius of  the loss class is bounded. Note
that if  \textbf{A\ref{assum:A2}} is met, then  so is \textbf{A\ref{assum:A4}}
necessarily. We can now state our main results.

\begin{theorem}[High probability oracular inequality]
  \label{thm:main:one}
  Suppose     that     \textbf{A\ref{assum:A0}},     \textbf{A\ref{assum:A1}},
  \textbf{A\ref{assum:A2}},            \textbf{A\ref{assum:A3}}            and
  \textbf{A\ref{assum:A4}} are met. Define
  \begin{equation}
    \label{eq:def:v2}
    v_{2}                                                                   :=
    \frac{3\pi}{2}\left[\left(\frac{15b_{2}}{|\calA|/\deg(\calG)}\right)^{2} +
      \frac{64v_{1}}{|\calA|/\deg(\calG)}\right].  
  \end{equation}
  Fix arbitrarily two integers $N,N' \geq 2$ and a real number $a>0$, then set
  $\underline{x}(a,N):=       a        [2^{-N}       v_{2}/\gamma]^{1/\beta}$,
  $\underline{x}'(a,N'):=  a  b_{1}   2^{-N'}$.   For  all  $t   \geq  1$  and
  $x \geq \underline{x}(a,N)$, it holds that
  \begin{multline}
    \label{eq:thm:main:one:a}
    \Prob                \left[\widetilde{R}_{\widehat{j}_t,t}               -
      \widetilde{R}_{t}(\theta^{\circ})               \geq              (1+2a)
      \left(\widetilde{R}_{\widetilde{j}_t,t} -
        \widetilde{R}_{t}(\theta^{\circ})\right) + x\right] \\
    \leq 2 JN  \left[\exp \left(- \frac{t x^{2-\beta}}{C_1(a)}  \right) + \exp
      \left( -\frac{t x}{C_2(a)} \right)\right],
  \end{multline}
  where   $C_1(a)   :=   2^{5-\beta}   (1    +   a)^2   \gamma   /   a^\beta$,
  $C_2(a)  :=   8  (1+a)  b_{2}/3$.   Moreover,   for  all  $t  \geq   1$  and
  $x \geq \underline{x}'(a,N')$, it also holds that
  \begin{multline}
    \label{eq:thm:main:one:b}
    \Prob                \left[\widetilde{R}_{\widehat{j}_t,t}               -
      \widetilde{R}_{t}(\theta^{\circ})               \geq              (1+2a)
      \left(\widetilde{R}_{\widetilde{j}_t,t} -
        \widetilde{R}_{t}(\theta^{\circ})\right) + x\right] \\
    \leq     2e^{2}    J     N'     \Bigg[    \exp     \left(-\frac{[|\calA|/(
        t^{\beta}\deg(\calG))]   x^{2-\beta}}{C_{1}'(a)}    \right)   +   \exp
    \left(-\frac{[|\calA|/\deg(\calG)]x}{C_{2}'(a)} \right)\Bigg],
  \end{multline}
  where  $C_{1}'(a)  :=  2^{6+2\beta}e^{2}  (1+a)^{2}  \gamma/a^{\beta}$,  and
  $C_{2}'(a) := 60e(1+a)b_{2}$.
  
\end{theorem}

We   derive   the   following   oracular  inequality   in   expectation   from
Theorem~\ref{thm:main:one}.

\begin{corollary}[Oracular inequality for the expected risk]
  \label{cor:main}
  Suppose     that     \textbf{A\ref{assum:A0}},     \textbf{A\ref{assum:A1}},
  \textbf{A\ref{assum:A2}},            \textbf{A\ref{assum:A3}}            and
  \textbf{A\ref{assum:A4}} are met.  For any $a \in]0,1]$, it holds that
  \begin{multline}
    \label{eq:cor:a}
    \Exp           \left[          \widetilde{R}_{\widehat{j}_t,t}           -
      \widetilde{R}_{t}(\theta^{\circ})        -        (1        +        2a)
      \left(\widetilde{R}_{\widetilde{j}_t, t} -
        \widetilde{R}_{t}(\theta^{\circ})\right) \right] \\
    \leq    3\left(\frac{C_{1}(a)}{t}   \log    (2JN)\right)^{1/(2-\beta)}   +
    \frac{2C_{2}(a)}{t} \log(2JN)
  \end{multline}
  provided that $N\geq 2$ is chosen so that
  \begin{equation}
    \label{eq:cond:a}
    N      \geq         \frac{\beta}{2-\beta} \frac{\log(t)     +
      \log(C_{3})}{\log(2)}
  \end{equation}
  where  $C_{3}:=  (v_{2}/\gamma)^{(2-\beta)/\beta}/(2^{5-\beta}\gamma)$  with
  $v_{2}$ given by~\eqref{eq:def:v2}.  Moreover, it also holds that
  \begin{multline}
    \label{eq:cor:b}
    \Exp           \left[          \widetilde{R}_{\widehat{j}_t,t}           -
      \widetilde{R}_{t}(\theta^{\circ})        -        (1        +        2a)
      \left(\widetilde{R}_{\widetilde{j}_t, t} -
        \widetilde{R}_{t}(\theta^{\circ})\right) \right] \\
    \leq      3\left(\frac{C_{1}'(a)}{|\calA|/(t^{\beta}\deg(\calG))}     \log
      (2JN')\right)^{1/(2-\beta)}   +   \frac{2C_{2}'(a)}{|\calA|/\deg(\calG)}
    \log(2JN')
  \end{multline}
  provided that $N'\geq 2$ is chosen so that
  \begin{equation}
    \label{eq:cond:b}
    N' \geq \frac{\beta}{2-\beta} \frac{\log(|\calA|/(t^{\beta}\deg(\calG))) + 
      \log(C_{3}')}{\log(2)} 
  \end{equation}
  where $C_{3}':= b_{1}/(2^{6+2\beta}e^{2}\gamma)$.
\end{corollary}

\subsection{Summary of Section~\ref{subsec:SL}}
\label{subsec:summary}

Given $J$  algorithms $\widehat{\theta}_{1}, \ldots,  \widehat{\theta}_{J}$ to
learn  $\theta^{\star}$ from  $(\bar{O}_{t})_{t\geq  1}$,  the one-step  ahead
sequential  Super Learner  is a  meta-algorithm that  learns, as  data accrue,
which      one      performs      best.      In      words,      for      each
$j \in  \llbracket J\rrbracket:=\{1,  \ldots, J\}$, $\widehat{\theta}_j$  is a
procedure  that,  for  every  $t  \geq 1$,  maps  $\bar{O}_t$  to  an  element
$\theta_{j,t}$ of $\Theta$.

Strictly speaking,  the one-step  ahead sequential Super  Learner really  is an
online algorithm if each of the $J$ algorithms is online, that is, if for each
$j\in\llbracket  J\rrbracket$  and $t\geq  1$,  the  making of  $\theta_{j,t}$
consists  in  an  update  of  $\theta_{j,t-1}$ based  on  newly  accrued  data
$\bar{O}_{t}$. If  that is not  the case, then the  Super Learner is  merely a
sequential algorithm, updated at every time step $t$.

The    (unknown)    $t$-specific    measure    of    performance    of    each
$\widehat{\theta_{j}}$, $\widetilde{R}_{j,t}$ \eqref{eq:risk},  takes the form
of  an   average  cumulative  risk   conditioned  on  the   observed  sequence
$(\bar{Z}_{t})_{t \geq 1}$  introduced in Assumption~\textbf{A\ref{assum:A0}}.
The (unknown) $t$-specific oracular meta-algorithm  is indexed by the oracular
$\widetilde{j}_{t} \in \llbracket J \rrbracket$ \eqref{eq:oracle}.

We  use the  (known)  $t$-specific  empirical counterpart  $\widehat{R}_{j,t}$
\eqref{eq:erisk} of $\widetilde{R}_{j,t}$ to estimate $\widetilde{j}_{t}$ with
the   (known)   $t$-specific  $\widehat{j}_{t}$   \eqref{eq:oSL}.    Algorithm
$\widehat{\theta}_{j}$  with  $j  =  \widehat{j}_{t}$ is  the  one-step  ahead
sequential Super Learner at time $t$.

The oracular  inequalities in Corollary~\ref{cor:main} have  a familiar flavor
for  whoever  is  interested  in   Super  Learning  or,  more  generally,  the
aggregation or stacking  of algorithms.  In essence, as more  data accrue, the
expected risk of  the one-step ahead sequential Super Learner  is smaller than
$(1+a)$  ($a$  chosen   small)  times  the  expected  risk   of  the  oracular
meta-algorithm   up  to   an   error   term  of   the   form  constant   times
$(\log(J\log(\calI^{2}))/\calI^{2})^{1/(2-\beta)}$  where  $\calI$ grows  like
the amount of information available (the constant $\beta \in ]0,1]$ appears in
one  of the  assumptions).  In  \eqref{eq:cor:a}, $\calI^{2}$  equals~$t$.  In
\eqref{eq:cor:b}, $\calI^{2}$ equals $|\calA|/(t^{\beta}\deg(\calG))$.  In the
next section we  show that if the ratio  $|\calA|/\deg(\calG)$ is sufficiently
large    (both   in    absolute   terms    and   relative    to   $t$)    (see
\eqref{eq:condition:two}), then  the oracular inequality  \eqref{eq:cor:b} can
be    sharper    than    the   oracular    inequality    \eqref{eq:cor:a}    in
Corollary~\ref{cor:main}, revealing that we managed  to leverage a large ratio
$|\calA|/\deg(\calG)$ in the face of a small $t$.

\subsection{Comments}
\label{subsec:SL:comments}

\paragraph*{Leveraging a large ratio $\boldsymbol{|\calA|/\deg(\calG)}$ in the
  face of a small $\boldsymbol{t}$.}

Our  results generalize  those of~\cite{Benkeser2018}  in two  aspects. First,
they do not require assumptions akin to their assumptions $A3$ and $A4$, which
are meant  to deal with  the randomness  at play in  $\widetilde{R}_{j,t}$ and
$\Var[\Delta^{\circ}\ell(O_{\alpha,t}, Z_{t}) |  Z_{t}, F_{t-1}]$.  Instead we
exploit      a      so-called       stratification      argument      inspired
by~\cite{Cesa-Bianchi2008}.  Second,  our results  leverage the fact  that, as
explained  at the  beginning  of  Section~\ref{sec:theory}, each  $t$-specific
observation    is    a   collection    $(O_{\alpha,t})_{\alpha\in\calA}$    of
$(\alpha,t)$-specific  data   points  with  a  large   amount  of  conditional
independence between  them, as  modelled by  the conditional  dependency graph
$\calG$.   Recall that  $\deg(\calG)$  equals  1 plus  the  maximum degree  of
$\calG$. The smaller is $\deg(\calG)$ the more conditional independence we can
rely on.
 
If one  chooses $N=N'$  in \eqref{eq:cor:a} and  \eqref{eq:cor:b}, then  it is
easy  to  check that  the  two  terms in  the  right-hand  side expression  of
\eqref{eq:cor:b} are  smaller than  their counterparts in  \eqref{eq:cor:a} if
and only if
\begin{equation}
  \label{eq:condition:one}
  t^{1+\beta}    \leq     \frac{|\calA|/\deg(\calG)}{2e^{2}8^{\beta}}    \quad
  \text{and} \quad t \leq \frac{|\calA|/\deg(\calG)}{45e/2}. 
\end{equation}
Furthermore, a simple sufficient  condition for \eqref{eq:condition:one} to be
met is
\begin{equation}
  \label{eq:condition:two}
  t^{1+\beta}    \leq     \frac{|\calA|/\deg(\calG)}{2e^{2}8^{\beta}}    \quad
  \text{and}     \quad      \frac{|\calA|}{\deg(\calG)}     \geq     24e\times
  \left(3/(2e)\right)^{1/\beta}.  
\end{equation}
Thus,      if     \eqref{eq:condition:two}      is     met      (note     that
$24e\times  3/(2e)  =  36   \geq  24e\times  (3/(2e))^{1/\beta}$  whatever  is
$\beta  \in  ]0,1]$)  and  if  we   make  the  following  (valid)  choices  in
Corollary~\ref{cor:main},
\begin{equation*}
  N=N' \geq \frac{\beta}{2-\beta} \frac{\log(|\calA|/(t^{\beta}\deg(\calG))) + 
    \log(C_{3}')                                                             +
    \left(\log[C_{3}/(2e^{2}8^{\beta}C_{3}')]\right)_{+}}{\log(2)}, 
\end{equation*}
then the  oracular inequalities \eqref{eq:cor:a} and  \eqref{eq:cor:b} for the
expected risk hold  true, the latter being sharper than  the former. In words,
our analysis  does take  advantage of the  fact that  $|\calA|/\deg(\calG)$ is
large in the face of $t$ being comparatively small.

\paragraph*{A few details on the proofs.}

Theorem~\ref{thm:main:one} notably  hinges on the  Fan-Grama-Liu concentration
inequality  for martingales~\cite[Theorem  3.10  in][]{Bercu2015}  and on  the
following  result, tailored  to our  needs  and derived  from a  concentration
inequality   for   sums   of   partly   dependent   random   variables   shown
by~\citet{Janson04}.   For each  $j\in\llbracket J\rrbracket$  and $t\geq  1$,
introduce the two $(j,t)$-specific averages of conditional variances
\begin{align}
  \label{eq:var}
  \var_{j,t}
  &:=    \frac{1}{|\calA|}     \sum_{\alpha\in\calA}
    \Var\left[\Delta^{\circ}\ell(\theta_{j,\tau-1})    (O_{\alpha,t}Z_{\alpha,t})
    \middle| Z_{\alpha,t}, F_{t-1}\right],\\
  \label{eq:tildevar}
  \widetilde{\var}_{j,t}
  &:=                       \frac{1}{t}                      \sum_{\tau=1}^{t}
    \Var\left[\Delta^{\circ}\bar{\ell}(\theta_{j,\tau-1})(\bar{O}_{\tau},
    \bar{Z}_{\tau}) \middle| \bar{Z}_{\tau}, F_{\tau-1}\right]. 
\end{align}

\begin{theorem}
  \label{thm:Bernstein}
  Suppose that \textbf{A\ref{assum:A2}}  and \textbf{A\ref{assum:A3}} are met.
  For each $j \in \llbracket J\rrbracket$, $\widetilde{\var}_{j,t} \leq v_{2}$
  almost  surely  (see  \eqref{eq:def:v2}  for  the  definition  of  $v_{2}$).
  Moreover, for any $V > 0$ and all $x\geq 0$, if
  \begin{equation}
    \label{eq:calF}
    \widetilde{\calF}_{V}    :   =    \left[\max_{\tau   \in    \llbracket t\rrbracket}
      \{\var_{j,\tau}\} \leq V\right], 
  \end{equation}
  then
  \begin{equation}
    \label{eq:thm:Bernstein}
    \Prob\left[|[\widehat{R}_{j,t}    -     \widehat{R}_{t}    (\theta^{\circ})]-
      [\widetilde{R}_{j,t}  -  \widetilde{R}_{t}  (\theta^{\circ})]|  \geq  x,
      \widetilde{\calF}_{V}\right]                                                    \leq
    \exp\left(2-\frac{[|\calA|/\deg(\calG)]x^{2}}{32e^{2}V                   +
        15eb_{2}x}\right).
  \end{equation}
\end{theorem}

Our  proof of  Theorem~\ref{thm:Bernstein}  consists in  deriving a  Rosenthal
inequality  from  Janson's   concentration  inequality  [\citeyear{Janson04}],
following  \citeauthor{Petrov95}'s  line  of proof  [\citeyear{Petrov95}],  in
using  a  convexity  argument,  then  in   applying  the  same  method  as  in
\citep[][Corollary  3(b)]{Dedecker01}  (inspired  by the  proof  of  Theorem~6
in~\citep{DLP84}).   Inequality~\eqref{eq:thm:Bernstein} plays  a key  role in
the  derivation of~\eqref{eq:cor:b}.   The fact  that  the first  term in  the
right-hand       side       expression      in~\eqref{eq:cor:b}       features
$|\calA|/(t^{\beta}\deg(\calG))$ and  not $|\calA|/\deg(\calG)$ may  be deemed
pessimistic but is inherent to our  scheme of proof.  Note that substituting a
sharp  Marcinkiewicz-Zygmund-like  inequality~\cite[][Theorem~2.9]{Rio09}  for
the convexity  argument that  leads to \eqref{eq:convexity:argument}  does not
solve the issue.

Furthermore, it is  noteworthy that our results extend seamlessly  to the case
that                              every                             expression
$\ell(\theta_{\tau-1})(O_{\alpha,\tau},         Z_{\alpha,\tau})$         with
$\theta_{\tau-1}$ $F_{\tau-1}$-measurable is replaced  by an expression of the
form
$\ell(\theta_{\tau-1})(O_{\alpha,\tau},         Z_{\alpha,\tau})        \times
\omega_{\tau}(O_{\alpha,\tau}, Z_{\alpha,\tau})$,  where $\omega_{\tau}$  is a
$F_{\tau-1}$-measurable weighting  function.  This  proves very useful  in the
context of  reinforcement learning,  allowing to  rely on  importance sampling
weighting.

\paragraph*{Comments on the assumptions.}

Assumptions         \textbf{A\ref{assum:A1}},        \textbf{A\ref{assum:A2}},
\textbf{A\ref{assum:A3}}, \textbf{A\ref{assum:A4}} are quite typical.  Like in
the context of the application motivating  our study, suppose for example that
each                $O_{\alpha,t}$                decomposes                as
$O_{\alpha,t}  := (Z_{\alpha,t},  X_{\alpha,t}, Y_{\alpha,t})\in  \calZ \times
\calX \times [-B,B] =: \calO$  where $X_{\alpha,t}\in\calX$ is a collection of
covariates,  $Z_{\alpha,t}\in\calZ$   is  a  fixed  summary   measure  thereof
(\textit{i.e.},  as   explained  earlier,   $Z_{\alpha,t}$  is   derived  from
$X_{\alpha,t}$ by  evaluating at $X_{\alpha,t}$  a known, fixed (in  $t\geq 1$
and  $\alpha\in\calA$) function),  and $Y_{\alpha,t}  \in[-B,B]$ is  a bounded
real-valued  outcome of  primary  interest.  Suppose  moreover  that, for  all
$\alpha\in \calA$ and  $t\geq 1$, the conditional law  of $Y_{\alpha,t}$ given
$(X_{\alpha,t}, Z_{\alpha,t})=(x,z)$,  $(Z_{\alpha',t})_{\alpha'\in\calA}$ and
$F_{t-1}$  admits the  conditional  density  $y\mapsto f^{\star}(y|x,z)$  with
respect  to some  measure $\lambda(dy)$  on  $[-B,B]$.  In  this context,  the
conditional  expectation  $y\mapsto \theta^{\star}(y|x,z)$  of  $Y_{\alpha,t}$
given $(X_{\alpha,t},  Z_{\alpha,t})=(x,z)$ (for  all $(x,z)$ in  the support,
under $\Prob$, of  any $(X_{\alpha,t}, Z_{\alpha,t})$) is  an eligible feature
of interest.

Let $\Theta$ be the set of  measurable functions on $\calX\times \calZ$ taking
their          values         in          $[-B,B]$.          Given          by
$\ell(\theta):   ((z,x,y),  z)\mapsto   (y   -   \theta(x,z))^{2}$  (for   all
$\theta     \in      \Theta$),     the     least-square      loss     function
$\ell:\Theta\to  \bbR^{\calO\times\calZ}$ satisfies  \textbf{A\ref{assum:A1}}.
In    addition,   we    can   choose    $\theta^{\circ}:=\theta^{\star}$   and
\textbf{A\ref{assum:A2}},   \textbf{A\ref{assum:A3}}   (with  $\beta=1$)   and
\textbf{A\ref{assum:A4}} are  met. The  fact that  \textbf{A\ref{assum:A3}} is
met  follows from  a  classical  argument of  strong  convexity recalled,  for
self-containedness, in Appendix~\ref{app:strong:convexity}.

\section{Anticipating the cost of natural disasters}
\label{sec:anticipating}

In  this  section,  we  apply  one-step ahead  sequential  Super  Learning  to
anticipate  the  cost   of  natural  disasters.   Section~\ref{subsec:context}
presents  the  context  and objective  in  details,  Section~\ref{subsec:data}
describes the  data that we exploit,  and Section~\ref{subsec:modeling} models
the problem  in the terms  of the theoretical  Section~\ref{sec:theory}. Then,
Section~\ref{subsec:implementation}  discusses   the  implementation   of  the
one-step  ahead  sequential  Super  Learner  and  Section~\ref{subsec:results}
presents and comments its results.

\subsection{More on the context and the objective}
\label{subsec:context}

To better  anticipate the  risks, CCR  has developed  an expertise  in natural
disasters  modeling.   Its  cat  models exploit  portfolios  and  claims  data
collected from CCR's cedents to enable  a better appreciation of the exposures
of CCR, of its cedents and of the French State.

The   natural   disasters   compensation   scheme  created   by   French   Law
n\textsuperscript{o}82-600  is triggered  when the  following three  necessary
conditions are met:
\begin{enumerate}
\item a  government decree declaring a  natural disaster must be  published in
  the French \href{https://www.journal-officiel.gouv.fr/}{Journal Officiel};
\item  the lost  and/or damaged  property must  be covered  by a  property and
  casualty insurance policy;
\item a causal  link must exist between the declared  natural disaster and the
  sustained loss and/or damage.
\end{enumerate}
The mayor of a city can request the government declaration of natural disaster
by sending a form to their prefect.   All over France, the prefects gather the
forms  and  send  them  to  the  relevant  Interministerial  Commission.   The
commission  examines all  requests  and delivers  the  declaration of  natural
disaster if additional  criteria are met. These criteria  characterize what is
considered  as a  natural disaster.   For  instance, for  drought events  (the
natural    catastrophes     that    we    focus    on,     also    known    as
\textit{subsidence\footnote{The process by  which land or buildings  sink to a
    lower  level.}  events}  in  the  literature, for  reasons  that the  next
paragraph clarifies), the criteria evaluate  the intensity of the drought.  It
is noteworthy that the criteria are  regularly updated by the commission -- we
shall  discuss further  this  point in  Section~\ref{sec:discussion}.  If  the
Interministerial  Commission delivers  a favorable  opinion, confirmed  by the
publication  in  the Journal  Officiel  of  a  government decree  declaring  a
disaster,  then  CCR  indemnifies  the  insurance  companies  once  they  have
indemnified the policyholders.

As revealed  earlier, we focus on  drought events.  Such events  are caused by
the  clay  shrinking  and  swelling  during  a  calendar  year  (and  must  be
distinguished  from  \textit{agricultural}  drought events).   Drought  events
entail cracks  on buildings, which can  be covered by an  insurance policy. In
order  to manage  the risks  inherent  in the  natural disasters  compensation
scheme,  CCR  must  anticipate  the  costs  generated  by  drought  events  in
particular. This  is crucial for  the pricing of  non-proportional reinsurance
treaties,  and  for reserving  (that  is  to  say, to  anticipate  forthcoming
payments).  The present study tackles the  challenge of predicting the cost of
drought events from a data set that we describe next.

\subsection{Data}
\label{subsec:data}

The  data set  that we  exploit  to predict  the  costs of  drought events  is
composed of  several data sets  of different  natures.  The data  are commonly
grouped in two classes, depending on whether they concern the natural disaster
itself  or any  of the  remaining relevant  characteristics that  complete the
description of the  financial impact of the natural disaster  on the insurance
industry.  We  choose to  group the  data in two  other classes,  depending on
whether they come from the cedents or from another source.

\paragraph*{Data from cedents.}

CCR  reinsures  90\%  of  the   French  natural  disasters  insurance  market.
Contractually, CCR's  cedents must share their  portfolios (\textit{i.e.}, the
location and characteristics of the insured  goods) and claims data. Thanks to
this mechanism, CCR has gathered a large  data set that runs from 1990 to this
day.

\paragraph*{Data from other sources.}

The data from cedents are enriched  with other data collected from four French
organizations.  The  National Institute  for Statistical and  Economic Studies
(INSEE) and  Geographic National  Institute (IGN)  provide information  on the
cities (population,  area, proportions of  buildings by years  of construction
for INSEE; tree  coverage rate for IGN).  The French  Geological Survey (BRGM)
provides a mapping of the  clay shrinkage-swelling hazards in France. Finally,
Météo France provides a soil wetness  index (SWI).  This last feature consists
of time  series of values (one  every decade) ranging between  -3.33 (very dry
soil) and 2.33 (very wet soil).  Figure~\ref{graph:SWI} presents five one-year
chunks of SWI time series.

\begin{figure}
  \centering%
  \includegraphics[scale=1]{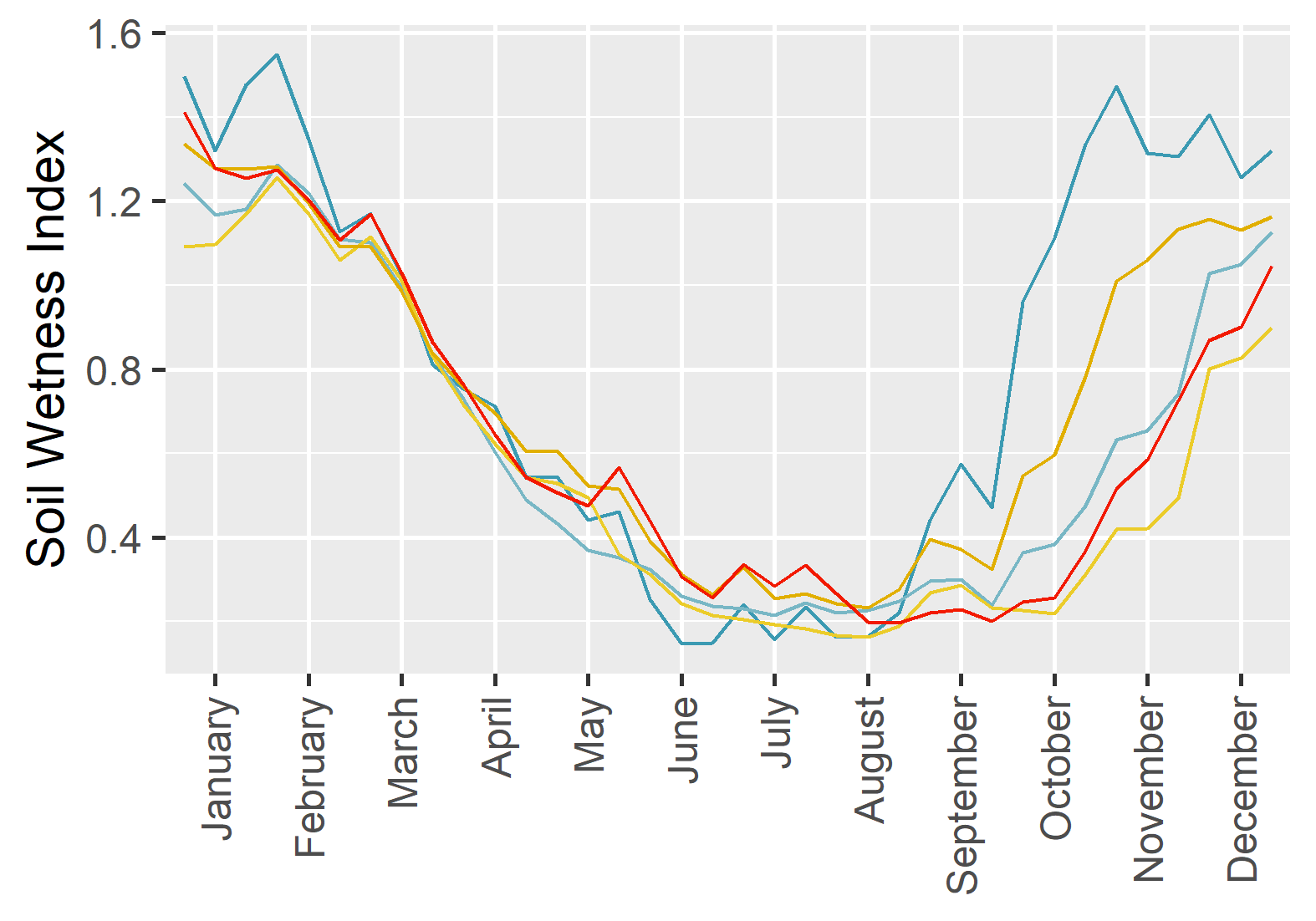}
  \caption{Chunks from  five arbitrarily  chosen time  series of  Soil Wetness
    Index (SWI) over the  course of one year.  It does not  come as a surprise
    that the soil is drier during summer than during winter.}
\label{graph:SWI}
\end{figure}

\paragraph*{Working at the city-level.}

It   is   noteworthy   that   the   spatial  resolution   of   SWI   data   is
$8 \times 8$km$^2$, which is much  larger than the 90\%-quantile of the French
cities area (30  km$^2$; only 1.3\% of  the French cities have  an area larger
than 65 km$^2$  -- data from 2014).   This issue will be  discussed further in
Section~\ref{sec:discussion}.   It justifies  why  we choose  to  work at  the
city-level as opposed  to the house level, by aggregating  at every time point
all data from each city into a single, time and city-specific observation.

\begin{itemize}
\item \textit{City-level  costs of drought  events.} For every time  point and
  each city,  the city-level  cost of drought  event is the  sum of  all house
  claims over the city's area.
\item  \textit{City-level  SWI.} For  every  time  point  and each  city,  the
  city-level SWI is the  convex average of the SWIs of  the $8 \times 8$km$^2$
  squares that overlap the city's area,  the weights being proportional to the
  areas of the intersections. 
\item \textit{City-level description.}  For every  time point and each city, a
  city-level description encapsulates the  city's profile.  The description is
  multi-faceted.   It contains:  an  indicator  of whether  or  not a  natural
  disaster was declared by the  government; the overall insured value obtained
  by summing the insured values over the  city's area; a summary of the city's
  clay hazard,  defined as the proportions  of houses falling in  each of four
  categories  of  clay   hazard;  a  summary  of  the   city's  dwelling  age,
  \textit{i.e.}, how  old houses  are, under  the form  of the  proportions of
  houses falling in each of four  categories; the climatic and seismic zones (a
  five-category  and  a four-category  variables);  a  summary of  the  city's
  vegetation; the city's number of houses, population, area, average altitude,
  and density, defined as  the ratio of the number of houses  to the area.  In
  addition, a  variety of features  are described by quantiles  that summarize
  distributions (\textit{e.g.},  the 30-quantiles  of the distribution  of the
  house-specific product of  SWI and insured value, or the  30-quantile of the
  distribution of the  house-specific product of the ground  slope and insured
  value, to mention just a few).  Overall, the city-level description consists
  of a little more than 430 variables.
\end{itemize}

\subsection{Modeling}
\label{subsec:modeling}

\paragraph{The sequence of observations.}

In the  context of  the anticipation  of the cost  of natural  disasters, each
$O_{\alpha,t}$                          decomposes                          as
$O_{\alpha,t} := (Z_{\alpha,t}, X_{\alpha,t}, Y_{\alpha,t})$ where
\begin{itemize}
\item $Y_{\alpha,t}  \in [0, B]$ is  the city-level cost of  the drought event
  for city $\alpha$ at year $t$,
\item $Z_{\alpha,t}$ is the city-level SWI for city $\alpha$ at year $t$,
\item  $X_{\alpha,t}$ is  the  city-level description  that encapsulates  city
  $\alpha$'s    profile    at    year     $t$,    including    an    indicator
  $W_{\alpha,t} \in \{0,1\}$  that equals 1 if and only  if a natural disaster
  has been declared by the government for that city and that year.
\end{itemize}
By  convention, $Y_{\alpha,t}  = 0$  if $W_{\alpha,t}  = 0$  (that is,  in the
absence  of  a declaration  of  natural  disaster).  Formally,  $X_{\alpha,t}$
includes  $Z_{\alpha,t}$.    For  notational   simplicity,  we   rewrite  each
$X_{\alpha,t}$                                                              as
$(W_{\alpha,t}, X_{\alpha,t},  Z_{\alpha,t}) \in  \{0,1\} \times  \calX \times
\calZ$, the  ``new'' $X_{\alpha,t}$ being the  ``old'' $X_{\alpha,t}$ deprived
of $Z_{\alpha,t}$ (but not of $W_{\alpha,t}$).

\paragraph{The feature of interest and related loss function.}

We assume that,  for all $\alpha\in \calA$ and $t\geq  1$, the conditional law
of  $Y_{\alpha,t}$ given  $(W_{\alpha,t},X_{\alpha,t}, Z_{\alpha,t})=(1,x,z)$,
$(Z_{\alpha',t})_{\alpha'\in\calA}$ and $F_{t-1}$ admits a conditional density
$y\mapsto f^{\star}(y|x,z)$ with respect to  some measure on $[0,B]$.  In this
context,  the  conditional  expectation  $y\mapsto  \theta^{\star}(y|x,z)$  of
$Y_{\alpha,t}$ given  $(W_{\alpha,t},X_{\alpha,t}, Z_{\alpha,t})=(1,x,z)$ (for
all $(x,z)$ in the support of any $(X_{\alpha,t}, Z_{\alpha,t})$ conditionally
on $W_{\alpha,t}=1$) is an eligible feature of interest.

Set $\calO := \{0,1\} \times \calX \times \calZ \times [0,B]$ and let $\Theta$
be the set of measurable functions  on $\calX\times \calZ$ taking their values
in  $[0,B]$   and  such   that  $\theta(x,z)   =  0$   if  $w=0$.    Given  by
$\ell(\theta): ((w,x,z,y), z)\mapsto (y - \theta(x,z))^{2} \1\{w=1\}$ (for all
$\theta     \in      \Theta$),     the     least-square      loss     function
$\ell:\Theta\to  \bbR^{\calO\times\calZ}$ satisfies  \textbf{A\ref{assum:A1}}.
In    addition,   we    can   choose    $\theta^{\circ}:=\theta^{\star}$   and
\textbf{A\ref{assum:A2}},   \textbf{A\ref{assum:A3}}   (with  $\beta=1$)   and
\textbf{A\ref{assum:A4}} are  met. The  fact that  \textbf{A\ref{assum:A3}} is
met  follows from  the  classical  argument of  strong  convexity recalled  in
Appendix~\ref{app:strong:convexity}.

\paragraph{Of $\boldsymbol{\calA}$ and $\boldsymbol{\calG}$.}

Here,  $\calA$ represents  the set  of  French cities.   The dependency  graph
$\calG$ used to  model the amount of  conditional independence operationalizes
two different types  of spatial dependence: one  \textit{geographical} and the
other  \textit{administrative}.   The  former corresponds  to  the  dependency
caused  by the  proximity between  two cities  in geological  and meteorological
terms  as well  as in  terms  of vegetation.   The latter  corresponds to  the
dependency  caused by  the administrative  proximity between  two cities  that
belong  to a  same  ``communauté de  communes''  (\textit{i.e.}, community  of
communes,  a  federation  of  municipalities). This  second  type  of  spatial
dependence is less obvious than the first  one. It arises from the fact that a
declaration of natural disaster must be requested  by the mayor of a city (see
Section~\ref{subsec:context}).  If, in a small  federation, a mayor makes such
a request, then it is likely that the other mayors will as well.

The  cardinality  of $\calA$  is  of  order  $36,000$.   In 2019,  there  were
approximately  $1,000$ federations  of  municipalities  in France,  regrouping
approximately  $26,000$  cities.  The federation  regrouped  approximately  30
cities in average.

\subsection{Implementation}
\label{subsec:implementation}

We implement our  statistical analysis in \texttt{R}~\citep{R}.  All our Super
Learners    are    implemented     based    on    the    \texttt{SuperLearner}
library~\citep{SuperLearner}.

\paragraph*{Base algorithms.}

The  base  algorithms   $\widehat{\theta}_{1},  \ldots,  \widehat{\theta}_{J}$
belong to one among four classes  of algorithms: the class of algorithms based
on  small to  moderate-dimensional working  models (linear  regression; lasso,
ridge  and  elastic   net  regressions~\citep{glmnet};  multivariate  adaptive
regression  splines~\citep{earth}; support  vector regression~\citep{kernlab};
gradient  boosting   with  linear  boosters~\citep{xgboost});  the   class  of
algorithms based  on trees (CART~\citep{rpart},  random forest~\citep{ranger},
gradient   boosting  with   tree  boosters~\citep{xgboost});   the  class   of
$k$-nearest   neighbors  algorithms;   the  class   of  algorithms   based  on
high-dimensional  neural networks~\citep{keras}.   Most of  the aforementioned
algorithms contribute several base algorithms  through the choice of different
hyper-parameters.  The $k$-nearest-neighbors  algorithms are customized.  Each
of them focuses on one of the quantiles summarizing a feature of interest (see
Section~\ref{subsec:data})  and  uses  the Kolmogorov-Smirnov  distance  as  a
measure of  similarity between every  pair of quantiles (viewed  as cumulative
distribution functions).

\paragraph*{Discrete and continuous one-step ahead sequential Super Learners.}

The one-step  ahead sequential Super  Learner that learns  $\theta^{\star}$ by
mapping $\bar{O}_{1},  \ldots, \bar{O}_{t}$ to  $\theta_{\widehat{j}_t,t}$ for
every $t \geq 1$ \eqref{eq:oSL} is known as a \textit{discrete} Super Learner.
The \textit{continuous}  Super Learner is the  \textit{discrete} Super Learner
when the  collection of  base algorithms consists  of all  convex combinations
$\sum_{j \in \llbracket J  \rrbracket} \sigma_{j} \widehat{\theta}_{j}$ of the
base  algorithms  $\widehat{\theta}_{1}, \ldots,  \widehat{\theta}_{J}$  where
$(\sigma_{1},  \ldots,  \sigma_{J})$  ranges   over  the  discretized  simplex
$\{(\sigma_{1},  \ldots,   \sigma_{J})  \in  \{(k-1)/K  :   k  \in  \llbracket
K+1\rrbracket\}^{J}  : \sum_{j\in\llbracket  J \rrbracket}  \sigma_{j} =  1\}$
with a large integer $K$. Note that the cardinality of this collection of base
algorithms is of order $\Theta(K^{J})$ and  much larger than $J$.  This is not
overly problematic because $J$  in \eqref{eq:cor:a} and \eqref{eq:cor:b} plays
a    role   through    $\log(J)/\calI^{2}$   with    $\calI^{2}   =    t$   or
$\calI^{2}  = |\calA|/(t^{\beta}  \deg(\calG))$, one  of them  at least  being
supposed large.

For     every    $t     \geq    1$,     the    $\sigma$-specific     algorithm
$\sum_{j \in  \llbracket J  \rrbracket} \sigma_{j}  \widehat{\theta}_{j}$ maps
$\bar{O}_{1},               \ldots,              \bar{O}_{t}$               to
$\sum_{j  \in  \llbracket  J  \rrbracket} \sigma_{j}  \theta_{j,t}$.   From  a
computational  point of  view, we  do not  use the  larger collection  of base
algorithms obtained by convex combination.  Instead, we directly solve
\begin{equation}
  \label{eq:continuous:SL}
  \argmin_{\sigma      \in      \Sigma}     \frac{1}{t}      \sum_{\tau=1}^{t}
  \bar{\ell}\left({\textstyle\sum_{j \in  \llbracket J  \rrbracket} \sigma_{j}
      \theta_{j,\tau-1}}\right) (\bar{O}_{\tau}, \bar{Z}_{\tau}) 
\end{equation}
(where  $\Sigma$  is  the  whole  simplex), which  can  be  interpreted  as  a
convexified version of \eqref{eq:oSL}.
 
\paragraph*{More one-step ahead sequential  Super Learners and the overarching
  one-step ahead sequential Super Learner.}

We propose and implement two more  extensions. The first extension builds upon
the interpretation of \eqref{eq:continuous:SL}  as the so called meta-learning
task   consisting    in   predicting   $Y_{\alpha,\tau}$   under    the   form
$\sum_{j    \in   \llbracket    J\rrbracket}   \sigma_{j}    \theta_{j,\tau-1}
(X_{\alpha,\tau},             Z_{\alpha,\tau})$            for             all
$\alpha\in\calA, 1 \leq \tau\leq t$.   We consider other meta-learning methods
$m$        to         predict        $Y_{\alpha,\tau}$         based        on
$\theta_{1,\tau-1}(X_{\alpha,\tau},         Z_{\alpha,\tau})$,         \ldots,
$\theta_{J,\tau-1}(X_{\alpha,\tau},                         Z_{\alpha,\tau})$,
$X_{\alpha,\tau},             Z_{\alpha,\tau}$             for             all
$\alpha\in\calA, 1 \leq \tau\leq t$.  Each meta-learning method $m$ yields its
own $m$-specific (discrete or continuous) Super Learner.

The  second  extension   builds  upon  the  first  one.    The  collection  of
$m$-specific  Super  Learners  can  be  considered as  a  collection  of  base
algorithms. This raises the question of  learning which one performs best.  To
answer this  question, we can  rely on  what we call  the \textit{overarching}
(discrete or continuous) Super Learner.

\paragraph*{Meta-learning methods and overarching meta-learning method.}

In  view  of the  four  classes  of base  algorithms  described  in the  first
paragraph of this  section, the meta-learning methods belong to  one among two
classes   of   methods:   the   class   of   methods   based   on   small   to
moderate-dimensional  working  models   (linear  regression  with  nonnegative
coefficients~\citep{nnls};      lasso,     ridge      and     elastic      net
regressions~\citep{glmnet};    support   vector    regression~\citep{kernlab};
gradient  boosting   with  linear  boosters~\citep{xgboost});  the   class  of
algorithms   based   on   trees   (extra    trees,   a   variant   of   random
forest~\citep{ranger}; gradient boosting  with tree boosters~\citep{xgboost}).
The overarching  Super Learner uses  the meta-learning method based  on linear
regression  with nonnegative  coefficients.   The  discrete overarching  Super
Learner selects  which among  the Super Learners  (viewed as  base algorithms)
performs best.  The  continuous overarching Super Learner  learns which convex
combination of the Super Learners (viewed as base algorithms) performs best.

Overall, we implement 27 base algorithms and 48 Super Learners.

\subsection{Results}
\label{subsec:results}

\begin{table}[htbp]
  \centering
  \begin{tabular}{cccccc}
    min.&1st qu.& median&mean&3rd qu.&max\\\hline
    23&162.5&607&1072.3&1921.5&4436
  \end{tabular}
  \caption{Quantiles and  mean of  the yearly  numbers of  cities for  which a
    declaration  of natural  disaster was  delivered  by the  government as  a
    result  of   a  drought  event.  The   time  series  runs  from   1995  to
    2017. Overall, we count 24,663 such declarations.}
  \label{tab:quantiles:declarations}
\end{table}

We observe  the time series $(\bar{O}_{t})_{t\geq  1}$ from 1995 to  2017.  We
also observe  the years 2018 and  2019 but do  not know yet the  city-level or
global costs  of drought events for  these two years. Overall  we count 24,663
observations  $\bar{O}_{\alpha,\tau}$  for  which  a  declaration  of  natural
disaster was delivered by  the government as a result of  a drought event. The
quantiles and mean of the yearly numbers  of cities for which a declaration of
natural      disaster       was      delivered      are       reported      in
Table~\ref{tab:quantiles:declarations}.

\begin{figure}[htbp]
  \centering%
  \includegraphics[scale=1]{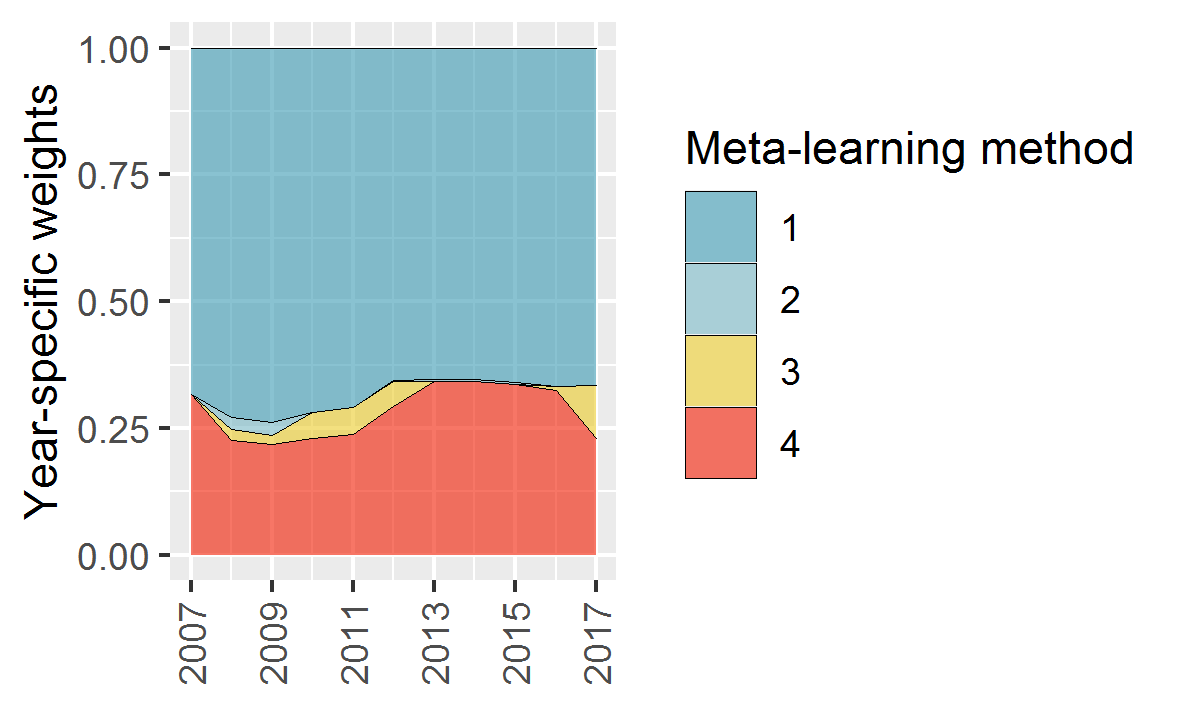}
  \caption{Evolution  (from 2007  onward)  of the  weights  attributed by  the
    overarching Super  Learner to 4 of  its base algorithms, each  one a Super
    Learner  itself  using  its  own meta-learning  method.   The  other  base
    algorithms get no weight at all (on this time window). }
\label{fig:overarching-weights}
\end{figure}

Among  the  48  Super  Learners,  the  overarching  continuous  Super  Learner
attributes positive weights to the  same four Super Learners consistently from
2007 to 2017, see Figure~\ref{fig:overarching-weights}. Moreover, the discrete
overarching Super Learner is consistently one of these four Super Learners.

\begin{figure}[htbp]
  \centering%
  \includegraphics[scale=1]{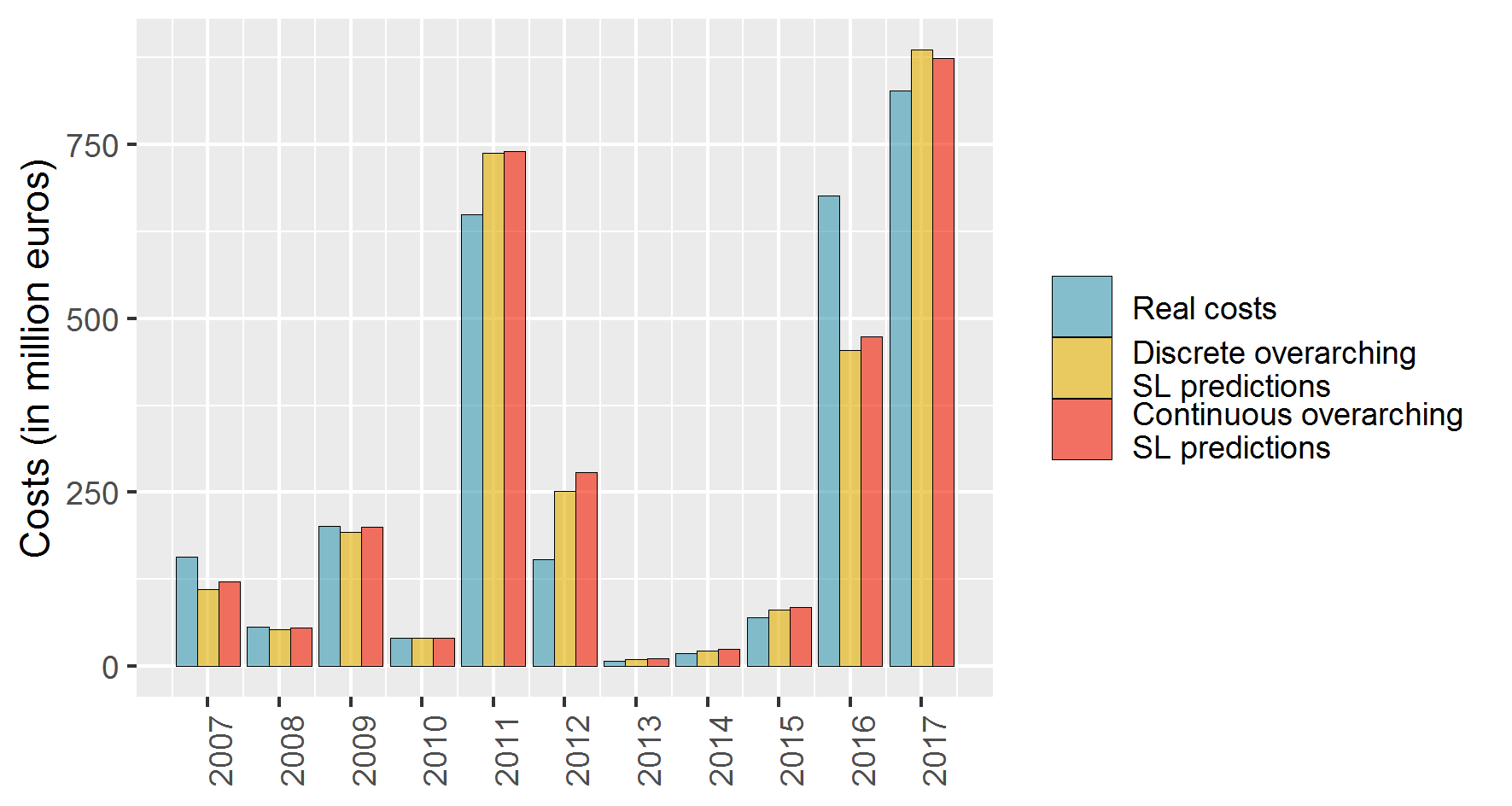}
  \caption{Presentation (from 2007 onward) of the real costs of drought events
    and  their predictions  made by  the discrete  and continuous  overarching
    Super Learners. }
\label{fig:predictions-overarching}
\end{figure}

Figure~\ref{fig:predictions-overarching}  represents   the  global   costs  of
drought events as  predicted by the discrete and  continuous overarching Super
Learners.   We observe  that  the discrete  and  continuous overarching  Super
Learners make  predictions that  are consistent each  year.  The  experts, who
naturally focus on the years for which the real costs happen to be the largest
because the financial stake is then the highest, deem them very satisfactory.

Overall, the averages (over the years) of the ratios of the predicted costs to
the  real costs  equal 106\%  (discrete overarching  Super Learner)  and 112\%
(continuous overarching Super  Learner). The ratios range  from 67\% (discrete
overarching Super Learner) and 70\% (continuous overarching Super Learner), in
2016,  to 164\%  (discrete overarching  Super Learner)  and 180\%  (continuous
overarching Super Learner), in 2012. The year  2016 is known by the experts to
be atypical, and challenging, because the average cost (understood here as the
ratio  of the  total cost  of the  year's drought  event to  the corresponding
number  of  declarations   of  natural  disaster  delivered   that  year)  is
particularly  large.  Conversely,  the  average  cost  in  the  year  2012  is
particularly small. 

\section{Discussion}
\label{sec:discussion}

We develop  and analyze a  meta-algorithm that  learns, as data  accrue, which
among $J$ base algorithms better learns  a feature $\theta^{\star}$ of the law
$\Prob$   of  a   sequence  $(\bar{O}_{\tau})_{\tau   \geq  1}$,   where  each
$\bar{O}_{\tau}$       consists       of       a       finite       collection
$(\bar{O}_{\alpha,\tau})_{\alpha \in \calA}$ of  many slightly dependent data.
We show  that the  meta-algorithm, an  example of  Super Learner,  leverages a
large ratio  $|\calA|/\deg(\calG)$ (a  measure of  the amount  of independence
among the $\tau$-specific $\bar{O}_{\alpha,\tau}$,  $\alpha \in \calA$) in the
face of a small number $t$ of time points where the time series is observed --
see  the  summary presented  in  Section~\ref{subsec:summary}.   The study  is
motivated  by the  challenge posed  by the  appreciation of  the exposures  to
drought events of CCR,  of its cedents and of the  French State.  We implement
and use two Super Learners to learn to assess the (global) costs by predicting
the (local) costs at the  city-level -- see Section~\ref{subsec:results} for a
summary of our results.

Reliable prediction of the cost of  a drought event \textit{must} rely on some
measure of  the drought's  intensity. We  exploit a  soil wetness  index (SWI)
provided by Météo  France. Because the spatial resolution of  SWI data is much
larger than the 90\%-quantile of the French  cities area, we choose to work at
the  city-level rather  than  at  the address-level,  by  aggregating all  the
address-specific  information at  the  city-level.  In  future  work, we  will
\textit{learn}  a  better  measure  of  the drought's  intensity  at  a  finer
resolution by  combining different  sources of  information pertaining  to the
soil wetness (SWI, rainfall,  nature of the soil, to name  just a few).  Since
we  know that  costs  can  vary dramatically  at  the  address-level, we  also
consider to later  try and enhance our  predictions by zooming in  back to the
address-level, thanks to the finer resolution measure of soil wetness.

In Section~\ref{subsec:context}, we explained that the criteria characterizing
what  is considered  as a  natural disaster  by the  relevant Interministerial
Commission are regularly updated.  Moreover, even  on the narrow time frame of
our study,  climate change may have  affected the severity of  droughts on the
French  territory.  From  a theoretical  viewpoint,  the marginal  law of  the
$(\alpha,\tau)$-specific  covariate   $Z_{\alpha,\tau}$  that   describes  the
severity  of  the   drought  depends  on  $\tau$.   We  tried   to  give  each
$O_{\alpha,\tau}$   an   $(\alpha,\tau)$-specific   weight   to   target   the
$(\alpha,t)$-specific  marginal  law  of $Z_{\alpha,t}$  when  addressing  the
prediction of the cost  of year $t$. If any, the benefits  were dwarfed by the
increase in variability caused by the learned weighting scheme.

\paragraph*{Acknowledgments.}

The  authors  thank  Jérôme  Dedecker  (MAP5, Université  de  Paris)  for  his
enlightened advice and  Thierry Cohignac (Caisse Centrale  de Réassurance) for
his suggestions to improve the manuscript.

\bibliography{osasl}

\appendix

\section{A classical strong convexity argument}
\label{app:strong:convexity}

Suppose   that   $\Theta$   is   convex,    and   that   the   loss   function
$\ell : \Theta \to \bbR^{\calO}$ is $a_{1}$-Lipschitz,
\begin{equation}
  \label{eq:strong:convexity:one}
  |\ell(\theta_{1}) -  \ell(\theta_{2})| \leq a_{1} |\theta_{1}  - \theta_{2}|
\end{equation}
and     $a_{2}$-strongly    convex:     for    all     $s\in    [0,1]$     and
$\theta_{1}, \theta_{2}\in \Theta$,
\begin{equation*}
  \ell(s\theta_{1}  +   (1-s)\theta_{2})  -  \tfrac{a_{2}}{2}   (s\theta_{1}  +
  (1-s)\theta_{2})^{2}       \leq       s       \left(\ell(\theta_{1})       -
    \left(\tfrac{a_{2}}{2}\theta_{1}\right)^{2}\right)         +         (1-s)
  \left(\ell(\theta_{2}) - \left(\tfrac{a_{2}}{2}\theta_{2}\right)^{2}\right) 
\end{equation*}
(both  inequalities above  are  understood pointwise).   Then  the modulus  of
continuity         of        $\ell$         is        lower-bounded         by
$\rho   \mapsto  \tfrac{a_{2}}{8}\rho^{2}$   in  the   sense  that,   for  all
$\theta_{1}, \theta_{2} \in \Theta$,
\begin{equation}
  \label{eq:strong:convexity:two}
  \tfrac{1}{2}     \left(\ell(\theta_{1})    +     \ell(\theta_{2})\right)    -
  \ell\left(\tfrac{1}{2}(\theta_{1}  + \theta_{2})\right)  \geq \tfrac{a_{2}}{8}
  (\theta_{1} - \theta_{2})^{2}
\end{equation}
(pointwise). Let  $P$ be a  law on $\calO$  such that $P\ell(\theta)$  is well
defined  for all  $\theta\in\Theta$,  where we  note  $Pf:=\int fdP$.   Choose
$\theta^{\circ}           \in            \Theta$           such           that
$P\ell(\theta^{\circ})  \leq P\ell(\theta)$  for all  $\theta\in\Theta$. Then,
for all $\theta\in\Theta$,
\begin{align*}
  \tfrac{1}{2} P(\ell(\theta) + \ell(\theta^{\circ}))
  &\geq P\ell(\tfrac{1}{2}(\theta  + \theta^{\circ}))  + \tfrac{a_{2}}{8}
    P(\theta - \theta^{\circ})^{2}\\
  &\geq P\ell(\theta^{\circ}) + \tfrac{a_{2}}{8}
    P(\theta - \theta^{\circ})^{2}\\
  &\geq        P\ell(\theta^{\circ})         +        \tfrac{a_{2}}{8a_{1}^{2}}
    P(\ell(\theta) - \ell(\theta^{\circ}))^{2},
\end{align*}
where the  first inequality follows from  \eqref{eq:strong:convexity:two}, the
second holds by convexity of $\Theta$  and choice of $\theta^{\circ}$, and the
third one follows from \eqref{eq:strong:convexity:one}. Therefore,
\begin{equation*}
  P(\ell(\theta^{\circ})  -  \ell(\theta))^{2} \leq  \tfrac{4a_{1}^{2}}{a_{2}}
  P(\ell(\theta) - \ell(\theta^{\circ})), 
\end{equation*}
which concludes the argument.

\section{Proofs}
\label{app:proofs}

\subsection{Proof of Theorem \ref{thm:main:one}}

The proof unfolds in three steps.

\paragraph{Step 1: an algebraic decomposition.}
For all $j \in \llbracket J\rrbracket$, $t \geq 1$ and $\theta\in \Theta$, let us define
\begin{align*}
  & \widetilde{H}_{j,t} := \widetilde{R}_{j,t} - \widetilde{R}_{t}(\theta^{\circ}),
    \quad         \widehat{H}_{j,t}        :=         \widehat{R}_{j,t}        -
    \widehat{R}_t(\theta^{\circ}), \quad \text{and}\\
  & \Delta^{\circ}\bar{\ell}(\theta) (\bar{O}_{t}, \bar{Z}_{t}) :=
    \bar{\ell}(\theta)(\bar{O}_{t},                \bar{Z}_{t})                -
    \bar{\ell}(\theta^{\circ})(\bar{O}_{t}, \bar{Z}_{t})
\end{align*}
($\bar{\ell}(\theta)$  is  defined  in~\eqref{eq:bar:ell}).   Fix  arbitrarily
$a > 0$.  An algebraic decomposition at  the heart of all studies of the Super
Learner~\cite[see,   \textit{e.g},][]{Dudoit2005,SL2007,Benkeser2018})  states
that    the    excess    risk    of    the    Super    Learner    (that    is,
$\widetilde{H}_{\widehat{j}_{t},t}$)  can be  bounded by  $(1+2a)$ times  the excess
risk  of  the oracle  (that  is,  $\widetilde{H}_{\widetilde{j}_{t},t}$), plus  some
remainder terms:
\begin{align}
  \notag
  \widetilde{H}_{\widehat{j}_{t},t} 
  &  \leq (1  + 2a)  \widetilde{H}_{\widetilde{j}_{t},t} +  A_{\widehat{j}_t,t}(a) +
    B_{\widetilde{j}_t,t}(a) \\
  \label{eq:SL_decomposition}
  &  \leq (1  +  2a) \widetilde{H}_{\widetilde{j}_{t},t}  +  \max_{j \in  \llbracket J\rrbracket}
    \{A_{j,t}(a)\} + \max_{j \in \llbracket J\rrbracket} \{B_{j,t}(a)\} 
\end{align}
where
\begin{equation*}
  A_{j,t}(a) := (1+a) \left(\widetilde{H}_{j,t}  - \widehat{H}_{j,t} \right) -
  a   \widetilde{H}_{j,t}  \quad   \text{and}   \quad   B_{j,t}(a)  :=   (1+a)
  \left(\widehat{H}_{j,t}     -     \widetilde{H}_{j,t}    \right)     -     a
  \widetilde{H}_{j,t}. 
\end{equation*}

The  first  terms in  the  definitions  of  $A_{j,t}(a)$ and  $B_{j,t}$  equal
$\pm(1+a)$ times
\begin{equation*}
  \frac{1}{t}\sum_{\tau=1}^{t}
  \Big(\Delta^{\circ}\bar{\ell}(\theta_{j,\tau-1}) 
  (\bar{O}_{\tau},                     \bar{Z}_{\tau})                     -
  \Exp\left[\Delta^{\circ}\bar{\ell}(\theta_{j,\tau-1})          (\bar{O}_{\tau},
    \bar{Z}_{\tau}) \middle| \bar{Z}_{\tau}, F_{\tau-1}\right]\Big),
\end{equation*}
that is  as the  average of  the $t$  first terms  of a  martingale difference
sequence.  As  for the shared second  term in the definitions  of $A_{j,t}(a)$
and $B_{j,t}$, it satisfies $-a\widetilde{H}_{j,t} \leq 0$. The second step of
the proof  consists in  exploiting two  so-called Bernstein's  inequalities to
control     the    probabilities     $\Prob[A_{j,t}(a)     \geq    x]$     and
$\Prob[B_{j,t}(a) \geq x]$ for $x\geq 0$.

\paragraph{Step   2:  Bounding   positive  deviations   of  $A_{j,t}(a)$   and
  $B_{j,t}(a)$.}

Set arbitrarily two integers  $N,N' \geq 2$ and a real number  $x \geq 0$. The
analysis  of  $\Prob[B_{j,t}(a) \geq  x]$  is  exactly  the  same as  that  of
$\Prob[A_{j,t}(a) \geq  x]$, so we  present only the  latter.  The key  to the
analysis     is    a     so-called     stratification    argument     inspired
by~\cite{Cesa-Bianchi2008}.

For  every  $j  \in  \llbracket  J\rrbracket$  and  $t  \geq  1$,  recall  the
definitions  \eqref{eq:var}   and  \eqref{eq:tildevar}  of   $\var_{j,t}$  and
$\widetilde{\var}_{j,t}$.  On  the one  hand, by  \textbf{A\ref{assum:A3}} and
because   the   functions   of   a  real   variable   $u\mapsto   u^{2}$   and
$u \mapsto u^{\beta}$ are respectively convex and concave, it holds that
\begin{align}
  \notag
  \widetilde{\var}_{j,t}
  & \leq \frac{1}{t}                    \sum_{\tau=1}^t
    \Exp\left[\left(\Delta^{\circ}\bar{\ell}(\theta_{j, \tau-1})(\bar{O}_{\tau},
    \bar{Z}_{\tau})\right)^{2} \middle| \bar{Z}_{\tau}, F_{\tau-1}\right]\\
  \label{eq:A3:csq}
  &\leq \gamma \left(\frac{1}{t} \sum_{\tau=1}^{t}
    \Exp\left[\Delta^{\circ}\bar{\ell}(\theta_{j,     \tau-1})(\bar{O}_{\tau},
    \bar{Z}_{\tau})                 \middle|                 \bar{Z}_{\tau},
    F_{\tau-1}\right]\right)^{\beta}                 =                \gamma
    \left(\widetilde{H}_{j,t}\right)^{\beta}
\end{align}
almost       surely.       Moreover,       it       also      holds       that
$\widetilde{\var}_{j,t}      \leq      v_{2}$       almost      surely      by
Theorem~\ref{thm:Bernstein}.  The  previous upper bound  and \eqref{eq:A3:csq}
play a key role in the first  version of Step~2 (Step~2 (v1)) presented below.
On the other hand,  by \textbf{A\ref{assum:A1}}, \textbf{A\ref{assum:A3}}, and
because the function  $u \mapsto u^{\beta}$ is concave it  holds almost surely
that, for all $\tau \in \llbracket t\rrbracket$,
\begin{align*}
  \notag
  \var_{j,\tau}
  &\leq  \frac{1}{|\calA|} \sum_{\alpha  \in  \calA}
    \Exp\left[\left(\Delta^{\circ}\ell(\theta_{j,\tau-1})(O_{\alpha,\tau},
    Z_{\alpha,\tau})  \right)^{2} \middle|  Z_{\alpha,\tau},
    F_{\tau-1}\right] \\
  &\leq \gamma \left(\Exp\left[\Delta^{\circ}\bar{\ell}(\theta_{j,\tau-1})(\bar{O}_{\tau},
    \bar{Z}_{\tau}) \middle| \bar{Z}_{\tau}, F_{\tau-1}\right]\right)^{\beta}.
\end{align*}
Consequently if $\widetilde{H}_{j,t} \leq B$  (an inequality that holds almost
surely when $B=b_{1}$, by \textbf{A\ref{assum:A2}}), then it also holds that
\begin{equation*}
  B \geq \widetilde{H}_{j,t}
  = \frac{1}{t}\sum_{\tau=1}^{t}
  \Exp\left[\Delta^{\circ}\bar{\ell}(\theta_{j,\tau-1})(\bar{O}_{\tau},
    \bar{Z}_{\tau}) \middle| \bar{Z}_{\tau}, F_{\tau-1}\right] \\
  \geq                      \frac{1}{t}                     \sum_{\tau=1}^{t}
  (\var_{j,\tau}/\gamma)^{1/\beta}. 
\end{equation*}
In summary we will use that, for any $B > 0$,
\begin{equation}
  \label{eq:A3:csq:bis}
  \1\left\{\widetilde{H}_{j,t}   \leq   B\right\}  \leq   \1\left\{\max_{1\leq
      \tau\leq t} \{\var_{j,\tau}\}  \leq \gamma(tB)^{\beta}\right\} =
  \1\{\widetilde{\calF}_{\gamma(tB)^{\beta}}\}
\end{equation}
($\widetilde{\calF}_{V}$  is defined  for any  $V>0$ in  \eqref{eq:calF}). The
upper bound $\widetilde{H}_{j,t} \leq  b_{1}$ and \eqref{eq:A3:csq:bis} play a
key role in the second version of Step~2 (Step~2 (v2)) presented below.

\begin{description}
\item[Step   2    (v1).]    Set   $v_{2}^{(-1)}    :=   0$   and,    for   all
  $i \in  \llbracket N-1\rrbracket$, $v_{2}^{(i)} :=  2^{i+1-N} \times v_{2}$.
  In         view        of         \eqref{eq:A3:csq}        and         since
  $\widetilde{\var}_{j,t}  \in  \cup_{i=0}^{N}  [v_{2}^{(i-1)},  v_{2}^{(i)}]$
  almost surely, it holds that
  \begin{align}
    \notag
    \Prob\left[A_{j,t}(a) \geq x\right]
    & = \Prob\left[\widetilde{H}_{j,t}  -  \widehat{H}_{j,t} \geq  \frac{1}{1+a}
      \left(x + a \widetilde{H}_{j,t}\right)\right]\\
    \notag
    & \leq \Prob\left[\widetilde{H}_{j,t} - \widehat{H}_{j,t} \geq \frac{1}{1+a}\left(x
      + a (\widetilde{\var}_{j,t}/\gamma)^{1/\beta}\right)\right]\\
    \notag
    & \leq \sum_{i=0}^{N-1} \Prob\left[\widetilde{H}_{j,t} - \widehat{H}_{j,t} \geq
      \frac{1}{1+a}\left(x   +   a   (\widetilde{\var}_{j,t}/\gamma)^{1/\beta}
      \right), \widetilde{\var}_{j,t} \in 
      [v_{2}^{(i-1)}, v_{2}^{(i)}] \right ] \\
  \label{eq:step2a:start}
    & \leq \sum_{i=0}^{N-1} \Prob\left[\widetilde{H}_{j,t} - \widehat{H}_{j,t} \geq
      \frac{1}{1+a}  \left(  x  +  a  (v_{2}^{(i-1)}/\gamma)^{1/\beta}  \right),
      \widetilde{\var}_{j,t} \leq v_{2}^{(i)} \right].
  \end{align}
  Note  that  $(\widetilde{H}_{j,t}  -  \widehat{H}_{j,t})_{t \geq  1}$  is  a
  martingale           adapted           to           the           filtration
  $(\sigma(F_{t},        \sigma(\bar{Z}_{t+1})))_{t\geq        1}$.         By
  \textbf{A\ref{assum:A2}} and the  Fan-Grama-Liu concentration inequality for
  martingales~\cite[Theorem  3.10   in][]{Bercu2015},  \eqref{eq:step2a:start}
  implies
  \begin{equation}
    \label{eq:step2a}
    \Prob\left[A_{j,t}(a)   \geq   x\right]    \leq   \sum_{i=0}^{N-1}   \exp\left(
      -\frac{1}{2} \frac{t D_i(x)}{(1+a)^2} \right), 
  \end{equation}
  where, for all $i \in \llbracket N-1\rrbracket$,
  \begin{equation*}
    D_i(x)             :=             \frac{\left(x            +             a
        (v_{2}^{(i-1)}  /  \gamma)^{1/\beta}\right)^2}{v_{2}^{(i)}  +  \frac{1}{3}
      \frac{b_{2}}{1 + a} \left(x + a (v_{2}^{(i-1)} / \gamma)^{1/\beta}\right)}.
  \end{equation*}
  Set      arbitrarily     $i      \in     \llbracket N-1\rrbracket\cup\{0\}$      and     define
  $x_i := 3(1+a)v_{2}^{(i)}/b_{2} - a (v_{2}^{(i-1)}/\gamma)^{1/\beta}$.
  \begin{itemize}
  \item          If           $x          \leq           x_{i}$,          then
    $v_{2}^{(i)} \geq (x + a (v_{2}^{(i-1)}/\gamma)^{1/\beta})\times b_{2} /(3
    (1+a))$ hence
    \begin{align}
      \notag
      D_{i}(x)
      &  \geq \frac{\left(x  + a  (v_{2}^{(i-1)}/\gamma)^{1/\beta}\right)^2}{2
        v_{2}^{(i)}}
        = \frac{\left(x  + a  (v_{2}^{(i-1)}/\gamma)^{1/\beta}\right)^{2-\beta}}{2
        v_{2}^{(i)}/\left(x                         +                        a
        (v_{2}^{(i-1)}/\gamma)^{1/\beta}\right)^\beta}\\
      \label{eq:Di:lower:bound:a1}
      & \geq
        \frac{x^{2        -\beta}}{2         v_{2}^{(i)}/\left(x        +        a
        (v_{2}^{(i-1)}/\gamma)^{1/\beta}\right)^\beta}. 
    \end{align}
    If $i\neq 0$, then \eqref{eq:Di:lower:bound:a1} entails
    \begin{equation}
      \label{eq:Di:lower:bound:a2}
      D_{i}(x)                 \geq                 \frac{x^{2-\beta}}{2\gamma
        v_{2}^{(i)}/(a^{\beta}v_{2}^{(i-1)})}       =       \frac{x^{2       -
          \beta}}{4\gamma/a^\beta}. 
    \end{equation}
    If $i=0$,  then \eqref{eq:Di:lower:bound:a2}  is also met  if and  only if
    $x \geq  \underline{x}(a, N)$, where  $\underline{x}(a, N)$ is  defined in
    the theorem.
  \item     Moreover      if     $x     \geq      x_{i}$,     then
    $v_{2}^{(i)} \leq (x + a (v_{2}^{(i-1)}/\gamma)^{1/\beta})\times b_{2} /(3
    (1+a))$ hence
    \begin{equation}
      \label{eq:Di:lower:bound:a3}
      D_i(x)            \geq            \frac{\left(x           +            a
          (v_{2}^{(i-1)}/\gamma)^{1/\beta}\right)^2}{\frac{2}{3}\frac{b_{2}}{1+a}\left(x
          + a (v_{2}^{(i-1)}/\gamma)^{1/\beta}\right)} = \frac{x           +            a
        (v_{2}^{(i-1)}/\gamma)^{1/\beta}}{\frac{2}{3}\frac{b_{2}}{1+a}}
      \geq       \frac{x}{\frac{2}{3} \frac{b_{2}}{1+a}}. 
    \end{equation}
  \end{itemize}
  Therefore,  in  light  of  \eqref{eq:step2a},  \eqref{eq:Di:lower:bound:a2},
  \eqref{eq:Di:lower:bound:a3}  and the  definitions  of $C_{1}(a),  C_{2}(a)$
  given in the theorem, for all $x \geq \underline{x}(a, N)$, it holds that
  \begin{align}
    \notag
    \Prob[A_{j,t}(a) \geq x]
    &  \leq  \sum_{i=0}{N-1}  \left[\1\{x \leq  x_i\}  \exp  \left(-\frac{t\times
      (2x)^{2-\beta}}{C_{1}(a)}  \right)  +  \1\{x  \geq  x_i\}  \exp
      \left( -\frac{t\times (2x)}{C_{2}(a)} \right)\right] \\
    \label{eq:lower:bound:step2a}
    &\leq  N \left[\exp  \left(-\frac{t \times  (2x)^{2-\beta} }{C_{1}(a)}
      \right) + 
      \exp\left(- \frac{t \times (2x)}{C_{2}(a)}\right)\right]. 
  \end{align}
\item[Step~2  (v2).]   This  step  is   very  similar  to  Step~2  (v1).   Set
  $b_{1}^{(-1)}  :=  0$  and,  for  all  $i  \in  \llbracket  N'-1\rrbracket$,
  $b_{1}^{(i)} := 2^{i+1-N'} \times  b_{1}$.  In view of \eqref{eq:A3:csq:bis}
  and                                                                    since
  $\widetilde{H}_{j,t}  \in  \cup_{i=0}^{N'-1}  [b_{1}^{(i-1)},  b_{1}^{(i)}]$
  almost surely, it holds that
  \begin{align}
    \notag
    \Prob\left[A_{j,t}(a) \geq x\right]
    & = \Prob\left[\widetilde{H}_{j,t}  -  \widehat{H}_{j,t} \geq  \frac{1}{1+a}
      \left(x + a \widetilde{H}_{j,t}\right)\right]\\
    \notag
    & \leq \sum_{i=0}^{N'-1} \Prob\left[\widetilde{H}_{j,t} - \widehat{H}_{j,t} \geq
      \frac{1}{1+a}\left(x +  a \widetilde{H}_{j,t} \right),  \widetilde{H}_{j,t} \in
      [b_{1}^{(i-1)}, b_{1}^{(i)}] \right ] \\
    \notag
    & \leq \sum_{i=0}^{N'-1} \Prob\left[\widetilde{H}_{j,t} - \widehat{H}_{j,t} \geq
      \frac{1}{1+a}  \left(  x  +  a  b_{1}^{(i-1)}  \right),
      \widetilde{H}_{j,t} \leq b_{1}^{(i)} \right]\\
    \label{eq:step2b:start}
    & \leq \sum_{i=0}^{N'-1} \Prob\left[\widetilde{H}_{j,t} - \widehat{H}_{j,t} \geq
      \frac{1}{1+a}  \left(  x  +  a  b_{1}^{(i-1)}  \right),
      \widetilde{\calF}_{\gamma(tb_{1}^{(i)})^{\beta}} \right].
  \end{align}
  By       \textbf{A\ref{assum:A2}}        and       \textbf{A\ref{assum:A3}},
  Theorem~\ref{thm:Bernstein} applies and \eqref{eq:step2b:start} yields
  \begin{equation}
    \label{eq:step2b}
    \Prob\left[A_{j,t}(a)   \geq   x\right]    \leq   \sum_{i=0}^{N'-1}   \exp\left(
      2-\frac{|\calA|/\deg(\calG)}{(1+a)^2}D_i'(x) \right), 
  \end{equation}
  where, for all $i \in \llbracket N'-1\rrbracket$,
  \begin{equation*}
    D_i'(x)             :=             \frac{\left(x            +             a
        b_{1}^{(i-1)}\right)^2}{32e^{2} \gamma(tb_{1}^{(i)})^{\beta} + 
      \frac{15eb_{2}}{1 + a} \left(x + a b_{1}^{(i-1)}\right)}.
  \end{equation*}
  Set     arbitrarily      $i     \in     \llbracket N'-1\rrbracket\cup\{0\}$      and     define
  $x_{i}'    :=    32e(1+a)\gamma   (tb_{1}^{(i)})^{\beta}/(15b_{2})    -    a
  b_{1}^{(i-1)}$.
  \begin{itemize}
  \item          If          $x           \leq          x_{i}'$,          then
    $32e^{2}\gamma(tb_{1}^{(i)})^{\beta}  \geq  (x  +  a  b_{1}^{(i-1)})\times
    15eb_{2} /(1+a)$ hence
    \begin{align}
      \notag
      D_{i}'(x)
      &     \geq     \frac{\left(x      +     a     b_{1}^{(i-1)}\right)^2}{64
        e^{2}\gamma(tb_{1}^{(i)})^{\beta}} 
        = \frac{\left(x  + a  b_{1}^{(i-1)}\right)^{2-\beta}}{64
        e^{2}\gamma(tb_{1}^{(i)})^{\beta}/\left(x + a 
        b_{1}^{(i-1)}\right)^\beta}\\
      \label{eq:Di:lower:bound:b1}
      & \geq
        \frac{x^{2        -\beta}}{64
        e^{2}\gamma(tb_{1}^{(i)})^{\beta}/\left(x + a 
        b_{1}^{(i-1)}\right)^\beta}. 
    \end{align}
    If $i\neq 0$, then \eqref{eq:Di:lower:bound:b1} entails
    \begin{equation}
      \label{eq:Di:lower:bound:b2}
      D_{i}'(x)                 \geq                 \frac{x^{2-\beta}}{64
        e^{2}\gamma(tb_{1}^{(i)})^{\beta}/(a 
        b_{1}^{(i-1)})^\beta}       =       \frac{x^{2       -
          \beta}}{64e^{2}\gamma(2t/a)^{\beta}}. 
    \end{equation}
    If $i=0$,  then \eqref{eq:Di:lower:bound:b2}  is also met  if and  only if
    $x \geq  \underline{x}'(a, N')$, where $\underline{x}'(a,  N')$ is defined
    in the theorem.
  \item       Moreover       if        $x       \geq       x_{i}'$,       then
    $32e^{2}\gamma(tb_{1}^{(i)})^{\beta}  \leq  (x  +  a  b_{1}^{(i-1)})\times
    15eb_{2} /(1+a)$ hence
    \begin{equation}
      \label{eq:Di:lower:bound:b3}
      D_i'(x)            \geq            \frac{\left(x           +            a
          b_{1}^{(i-1)}\right)^2}{\frac{30eb_{2}}{1   +   a}   \left(x   +   a
          b_{1}^{(i-1)}\right)} = \frac{x + a 
        b_{1}^{(i-1)}}{\frac{30eb_{2}}{1   +   a}}
      \geq       \frac{x}{\frac{30eb_{2}}{1   +   a}}. 
    \end{equation}
  \end{itemize}
  Therefore,  in  light  of  \eqref{eq:step2b},  \eqref{eq:Di:lower:bound:b2},
  \eqref{eq:Di:lower:bound:b3} and  the definitions of  $C_{1}'(a), C_{2}'(a)$
  given in the theorem, for all $x \geq \underline{x}'(a, N')$, it holds that
  \begin{align}
    \notag
    \Prob[A_{j,t}(a) \geq x]
    &     \leq     \sum_{i=0}^{N'-1}     \Bigg[1\{x    \leq     x_i'\}     \exp
      \left(2-\frac{[|\calA|/( t^{\beta}\deg(\calG))]\times 
      (2x)^{2-\beta}}{C_{1}'(a)} \right) \\
    \notag
    &\qquad\qquad + \1\{x \geq x_i'\} \exp
      \left(2 -\frac{[|\calA|/\deg(\calG)]\times (2x)}{C_{2}'(a)} \right)\Bigg] \\
    \notag
    &\leq  N'   \Bigg[ \exp
      \left(2-\frac{[|\calA|/( t^{\beta}\deg(\calG))]\times 
      (2x)^{2-\beta}}{C_{1}'(a)} \right)\\
    \label{eq:lower:bound:step2b}
    &\qquad \qquad \qquad + \exp
      \left(2 -\frac{[|\calA|/\deg(\calG)]\times (2x)}{C_{2}'(a)} \right)\Bigg]. 
  \end{align}
\end{description}

\paragraph{Step 3: end of the proof.}

In view of \eqref{eq:SL_decomposition}, a union bound implies that
\begin{align*}
  \Prob
  \left[     \widetilde{H}_{\widehat{j}_{t},t}     -    (1     +     2a)
  \widetilde{H}_{\widetilde{j}_{t}, t} \geq x \right] 
  &  \leq \Prob  \left[ \max_{j  \in \llbracket J\rrbracket}  \{A_{j,t}(a)\} +  \max_{j \in  \llbracket J\rrbracket}
    \{B_{j,t}(a)\} \geq x \right] \\ 
  &  \leq \sum_{j=1}^{J}  \left(\Prob \left[A_{j,t}(a)  \geq x/2  \right] +  P
    \left[B_{j,t}(a) \geq x/2 \right]\right).
\end{align*}
For all  $x \geq  \underline{x}(a,N)$, \eqref{eq:thm:main:one:a}  follows from
\eqref{eq:lower:bound:step2a}    and   the    above   inequality;    for   all
$x   \geq   \underline{x}'(a,N')$,  \eqref{eq:thm:main:one:b}   follows   from
\eqref{eq:lower:bound:step2b} and  the above  inequality.  This  completes the
proof of Theorem~\ref{thm:main:one}.~\hfill$\square$

\subsection{Proof of Corollary~\ref{cor:main}}

Corollary~\ref{cor:main}  follows  from the  straightforward  application,
twice, of the next technical  lemma, based on \eqref{eq:thm:main:one:a} on the
one hand and on \eqref{eq:thm:main:one:b} on the other hand.

\begin{lemma}
  \label{lem:tech:lemma}
  Let     $a,b,c>0$,     $\beta\in]0,1]$     be     some     constants     and
  $(\underline{x}(N))_{N\geq  2}$  be  a  sequence of  positive  numbers  that
  decreases  to  0.  Let  $U$  be  a real  valued  random  variable such  that
  $E[|U|]   <  \infty$   and,   for   all  integer   $N   \geq   2$  and   all
  $x \geq \underline{x}(N) > 0$,
  \begin{equation}
    \label{eq:tech:lemma:hyp}
    \Prob[U \geq  x] \leq a N  \left[\exp(-x^{2-\beta}/b) +
      \exp(-x/c)\right].
  \end{equation}
  If
  $N \geq  \min\{n \geq  2 :  \underline{x}(n) \leq  b^{1/(2-\beta)}, \log(an)
  \geq 1\}$, then
  \begin{equation}
    \label{eq:tech:lemma:res}
    \Exp[U] \leq 3 (b\log(aN))^{1/(2-\beta)} + 2c \log(aN).
  \end{equation}
\end{lemma}

\begin{proof}[Proof of Lemma~\ref{lem:tech:lemma}]
  It is well known that
  \begin{equation*}
    \Exp[U] \leq  \Exp[U\1\{U\geq 0\}]  = \int_{0}^{\infty}  \Prob[U \1\{U\geq
    0\} \geq x] dx = \int_{0}^{\infty}     \Prob[U \geq x] dx. 
  \end{equation*}
  Therefore, for any $N \geq 2$,
  \begin{align}
    \notag
    \Exp[U]
    & \leq \int_{0}^{\infty} \left(\1\{x  < \underline{x}(N)\} + \1\{x \geq
      \underline{x}(N)\} a N  \left[\exp(-x^{2-\beta}/b) +
      \exp(-x/c)\right]\right) dx\\
    \label{eq:Exp:U}
    & \leq  \underline{x}(N) +  \int_{0}^{\infty} \min\{1,  a N
      \exp(-x^{2-\beta}/b)\} dx + \int_{0}^{\infty} \min\{1,  a N
      \exp(-x/c)\} dx.
  \end{align}
  Let us  denote by $\LL$  and $\RR$ the  above left-hand side  and right-hand
  side                            integrals.                            Choose
  $N \geq  \min\{n \geq  2 :  \underline{x}(n) \leq  b^{1/(2-\beta)}, \log(an)
  \geq 1\}$.
  \begin{description}
  \item[Bounding    $\LL$.]     Let    $x_{\LL}$    be    chosen    so    that
    $a    N    \exp(-x_{\LL}^{2-\beta}    /   b)    =    1$,    \textit{i.e.},
    $x_{\LL}  :=  (b\log(aN))^{1/(2-\beta)}$. Now,  thanks  to  the change  of
    variable $u =  x^{2-\beta}/b$ and because $u \mapsto  u^{1/(2-\beta) - 1}$
    is nonincreasing,
    \begin{align}
      \notag
      \LL
      &= x_{\LL} + aN \int_{x_{\LL}}^{\infty} \exp ( -x^{2-\beta}/b) dx \\
      \notag
      &= x_{\LL} + aN b^{1/(2-\beta)} \int_{\log(aN)}^{\infty} \exp(-u) u^{1/(2-\beta)-1} du \\
      \notag
      &\leq     x_{\LL}+     \frac{aN     (b\log(aN))^{1/(2-\beta)}}{\log(aN)}
        \int_{0}^{\infty} \exp(-u) du \\
      \label{eq:LL}
      &= x_{\LL} (1 + 1/\log(aN)) \leq 2 (b\log(aN))^{1/(2-\beta)}.
    \end{align}
  \item[Bounding    $\RR$.]     Let    $x_{\RR}$    be    chosen    so    that
    $aN \exp(-x_{\RR} / c) =1$, \textit{i.e.},  $x_{\RR} := c \log(aN)$. It is
    readily seen that
    \begin{equation}
      \label{eq:RR}
      \RR
      = x_{\RR} + aN \int_{x_{\RR}}^{\infty} \exp(-x/c) dx 
      = x_{\RR} + acN \exp(-x_{\RR}/c) = c (1 + \log(aN)). 
    \end{equation}
  \end{description}
  In view of \eqref{eq:Exp:U}, \eqref{eq:LL},  \eqref{eq:RR}, and by choice of
  $N$, we obtain
  \begin{align*}
    \Exp[U]
    & \leq b^{1/(2-\beta)} + 2 (b\log(aN))^{1/(2-\beta)} + c (1 + \log(aN))\\
    & \leq 3 (b\log(aN))^{1/(2-\beta)} + 2c \log(aN).
  \end{align*}
  This completes the proof.
\end{proof}

Set  $t \geq  1$  and  $a \in]0,1]$.   In  view of  \eqref{eq:thm:main:one:a},
Lemma~\ref{lem:tech:lemma}  yields   \eqref{eq:cor:a}  under   the  sufficient
condition that $N\geq 2$ also satisfy \eqref{eq:cond:a}.  Moreover, in view of
\eqref{eq:thm:main:one:b},     Lemma~\ref{lem:tech:lemma}      also     yields
\eqref{eq:cor:b} under  the sufficient condition  that $N\geq 2$  also satisfy
\eqref{eq:cond:b}.  This completes the proof of the corollary.~\hfill$\square$

\subsection{Proof of Theorem~\ref{thm:Bernstein}}
\label{sec:Bernstein}

The   proof  of   Theorem~\ref{thm:Bernstein}  hinges   on  a   Bernstein-like
concentration inequality for  sums of partly dependent  random variables shown
by~\citet[][Theorem~2.3]{Janson04}.  Janson  emphasizes that his  theorem uses
the      independence      of       suitable      (large)      subsets      of
$(\zeta_{\alpha})_{\alpha  \in  \calA}$,  not  any other  information  on  the
dependencies, so that the result must be suboptimal when the dependencies that
exist are weak.  We recall the theorem for completeness.

\begin{theorem}[\citet{Janson04}]
  \label{thm:Janson}
  Let  $(\zeta_{\alpha})_{\alpha  \in  \calA}$   be  a  collection  of  random
  variables      with     dependency      graph     $\calG$      such     that
  $\zeta_{\alpha} -  \Exp[\zeta_{\alpha}] \leq  B$ for  some $B  > 0$  and all
  $\alpha                 \in                  \calA$.                  Define
  $\Zeta:=   |\calA|^{-1}\sum_{\alpha    \in   \calA}    \zeta_{\alpha}$   and
  $V := |\calA|^{-1} \sum_{\alpha \in \calA} \Var[\zeta_{\alpha}]$.  Then, for
  all $x \geq 0$,
  \begin{equation}
    \label{eq:Janson}
    \Prob\left[\Zeta      -      \Exp[\Zeta]      \geq      x\right]      \leq
    \exp\left(-\frac{|\calA|V}{B^{2}                              \deg(\calG)}
      h\left(\frac{4Bx}{5V}\right)\right), 
  \end{equation}
  where $h:u \mapsto (1+u) \log(1+u) - u$.
\end{theorem}

Note  that  \eqref{eq:thm:Bernstein}   from  Theorem~\ref{thm:Bernstein}  also
writes as
\begin{equation*}
  \label{eq:thm:Bernstein:alt}
  \Prob\left[|\widehat{H}_{j,t}     -     \widetilde{H}_{j,t}|     \geq     x,
    \widetilde{\calF}_{V}\right]                                          \leq
  \exp\left(2-\frac{[|\calA|/\deg(\calG)]x^{2}}{32e^{2}V + 15eb_{2}x}\right). 
\end{equation*}
Following  the  line  of  proof of  the  Rosenthal  inequality  by~\citet[page
59]{Petrov95} (see also the proof of  Theorem 5.2 in \citep{Baraud00}), we use
\eqref{eq:Janson}   to  control   $\Exp[|\Zeta   -  \Exp[\Zeta]|^{p}]$   hence
$\Exp[|\widehat{H}_{j,t} -  \widetilde{H}_{j,t}|^{p}]$ (by convexity)  for all
$p \geq 2$.  The bound~\eqref{eq:thm:Bernstein} follows as in~\cite[][proof of
Corollary  3(b)]{Dedecker01}, a  method  inspired by  the  proof of  Theorem~6
in~\citep{DLP84}.

We first prove this corollary  of Theorem~\ref{thm:Janson}.  The constants are
in no way optimal.

\begin{corollary}
  \label{cor:Janson}
  In the context of Theorem~\ref{thm:Janson}, for all $p\geq 2$,
  \begin{equation}
    \label{eq:moments}
    \Exp\left[|\Zeta    -    \Exp[\Zeta]|^{p}\right]    \leq    \frac{3\pi}{2}
    \left[\left(\frac{15B\deg(\calG)}{2|\calA|}\right)^{p}       p^{p}       +
      \left(\frac{32V\deg(\calG)}{|\calA|}\right)^{p/2}       p^{p/2}\right].  
  \end{equation}
\end{corollary}

\begin{proof}[Proof of Corollary~\ref{cor:Janson}]  Fix arbitrarily $p\geq 2$.
  It               is               well              known               that
  $\Exp[U^{p}]  = \int_{0}^{\infty}  p s^{p-1}  \Prob[U  \geq s]  ds$ for  any
  nonnegative random  variable $U$.  Let  $r > 0$ be  a constant that  we will
  carefully choose later on. Set arbitrarily  $s \geq 0$, define $m:=s/r$, and
  introduce
  \begin{equation*}
    \tZeta_{m}  :=  |\calA|^{-1}\sum_{\alpha \in  \calA}  (\zeta_{\alpha}  -
    \Exp[\zeta_{\alpha}]) \1\{|\zeta_{\alpha} - \Exp[\zeta_{\alpha}]| < m\}.
  \end{equation*}
  It holds that
  \begin{align*}
    \Prob(|\Zeta - \Exp[\Zeta]| \geq s)
    &  \leq \Prob[\Zeta -  \Exp[\Zeta] \neq
      \tZeta_{m}] + \Prob[|\Zeta - \Exp[\Zeta]| \geq s, \Zeta - \Exp[\Zeta] = \tZeta_{m}] \\
    & \leq
      \Prob[r\max_{\alpha  \in  \calA} |\zeta_{\alpha}  -  \Exp[\zeta_{\alpha}]|
      \geq s] + \Prob[|\Zeta - \Exp[\Zeta]| \geq s, \Zeta - \Exp[\Zeta] = \tZeta_{m}]\\
    & \leq
      \Prob[r\max_{\alpha  \in  \calA} |\zeta_{\alpha}  -  \Exp[\zeta_{\alpha}]|
      \geq s] + \Prob[|\tZeta_{m} - \Exp[\tZeta_{m}]| \geq s - \Exp[\tZeta_{m}]]
  \end{align*}
  hence
  \begin{equation}
    \label{eq:one}
    \Exp[|\Zeta  - \Exp[\Zeta]|^{p}]  \leq r^{p}  \Exp[\max_{\alpha \in
      \calA}   |\zeta_{\alpha}  -   \Exp[\zeta_{\alpha}]|^{p}]  +   \int_{0}^{\infty}
    ps^{p-1} \Prob[|\tZeta_{m} - \Exp(\tZeta_{m})] \geq s - \Exp[\tZeta_{m}]]ds.
  \end{equation}
  We now note that
  \begin{align*}
    |\Exp[\tZeta_{m}]|
    & = |\Exp[\tZeta_{m} - (\Zeta - \Exp[\Zeta])]| \\
    & = |\calA|^{-1}
      \left|\Exp\left[\sum_{\alpha     \in      \calA}     (\zeta_{\alpha}     -
      \Exp[\zeta_{\alpha}]) \1\{|\zeta_{\alpha} - \Exp[\zeta_{\alpha}]| \geq
      m\} \right] \right|\\
    &   \leq  (m|\calA|)^{-1}  \sum_{\alpha   \in  \calA}
      \Var[\zeta_{\alpha}] = V/m.
  \end{align*}
  Therefore if  $s \geq s_{0}  := \sqrt{2rV}$, then  $s/2 \geq V/(s/r)  = V/m$
  hence $s - |\Exp[\tZeta_{m}]| \geq  s/2$.  In light of \eqref{eq:Janson} and
  \eqref{eq:one}, the rightmost term in \eqref{eq:one}, say $I_{p}$, satisfies
  \begin{align}
    \notag I_{p}
    &\leq   \int_{0}^{s_{0}}    ps^{p-1}ds   +
      \int_{s_{0}}^{\infty} ps^{p-1} \Prob[|\tZeta_{m} - \Exp[\tZeta_{m}]| \geq s/2]ds\\
    \label{eq:two}
    &\leq  s_{0}^{p}   +  2   \int_{s_{0}}^{\infty}  ps^{p-1}
      \exp\left(-\frac{|\calA|\tilde{V}}{4m^{2}                                \deg(\calG)}
      h\left(\frac{8ms/2}{5\tilde{V}}\right)\right) ds,
  \end{align}
  where
  \begin{align*}
    \tilde{V}
    &:= |\calA|^{-1}\sum_{\alpha  \in \calA} \Var[(\zeta_{\alpha} -
      \Exp[\zeta_{\alpha}]) \1\{|\zeta_{\alpha} -  \Exp[\zeta_{\alpha}]| \leq m\}]\\
    &     \leq |\calA|^{-1}    \sum_{\alpha      \in     \calA}     \Exp[(\zeta_{\alpha}     -
      \Exp[\zeta_{\alpha}])^{2}  \1\{|\zeta_{\alpha}  -  \Exp[\zeta_{\alpha}]|  \leq
      m\}] \\
    &   \leq  |\calA|^{-1}  \sum_{\alpha   \in   \calA}   \Exp[(\zeta_{\alpha}   -
      \Exp[\zeta_{\alpha}])^{2}] = V.
  \end{align*}
  Because  $h(u)  \geq   \frac{u}{2}  \log(1  +  u)$  for  all   $u  \geq  0$,
  \eqref{eq:two} yields
  \begin{align}
    \notag I_{p}
    &\leq   s_{0}^{p}  +  2  \int_{s_{0}}^{\infty}
      ps^{p-1}      \exp\left(-\frac{|\calA|s}{10m      \deg(\calG)}     \log\left(1      +
      \frac{4ms}{5\tilde{V}}\right)\right) ds\\
    \label{eq:three}
    &=  s_{0}^{p}   +   2  \int_{s_{0}}^{\infty}   ps^{p-1}
      \exp\left(-\frac{|\calA|r}{10           \deg(\calG)}          \log\left(1           +
      \frac{4s^{2}}{5r\tilde{V}}\right)\right) ds.
  \end{align}
  If              $u:=             s/(5r\tilde{V}/4)^{1/2}$,              then
  $s^{p-1} \leq (5r\tilde{V}/4)^{(p-1)/2} (1 + u^{2})^{(p-1)/2}$.  A change of
  variable and the bound $\tilde{V} \leq V$ thus imply that the rightmost term
  in \eqref{eq:three} is smaller than
  \begin{equation*}
    2p \left(\frac{5rV}{4}\right)^{p/2} \int_{0}^{\infty} (1 +
    u^{2})^{(p-1)/2 - r|\calA|/(10\deg(\calG))} du.
  \end{equation*}
  We now choose $r :=  5(p+1)\deg(\calG)/|\calA|$ to guarantee the convergence
  of  the above  integral, to  $\pi/2$, and  conclude that  \eqref{eq:one} and
  \eqref{eq:three} imply
  \begin{align}
    \notag
    \Exp[|\Zeta - \Exp[\Zeta]|^{p}]
    &\leq  r^{p} \Exp[\max_{\alpha
      \in    \calA}    |\zeta_{\alpha}    -    \Exp[\zeta_{\alpha}]|^{p}]    +
      \pi (rV)^{p/2} (2^{p/2} + p (5/4)^{p/2})\\
    \label{eq:four}
    &\leq (rB)^{p} + \pi (p+1) (2rV)^{p/2}.
  \end{align}
  Finally, since $(p+1)/p  \leq 3/2$ and $p^{2/p} \leq  e^{2/e} \approx 2.61$,
  we can  simplify \eqref{eq:four} to \eqref{eq:moments},  thus completing the
  proof of Corollary~\ref{cor:Janson}.
\end{proof}

Fix arbitrarily  $j \in \llbracket J\rrbracket$,  $t \geq 1$,  $V > 0$,  and $p\geq 2$.   To save
space introduce, for each $\tau \in \llbracket t\rrbracket$,
\begin{equation*}
  \Zeta_{j,\tau} := \Delta^{\circ}\bar{\ell}(\theta_{j,\tau-1})(\bar{O}_{\tau},
  \bar{Z}_{\tau})                                                        -
  \Exp\left[\Delta^{\circ}\bar{\ell}(\theta_{j,\tau-1})(\bar{O}_{\tau},
    \bar{Z}_{\tau})                \middle|                \bar{Z}_{\tau},
    F_{\tau-1}\right].
\end{equation*}
In     view     of     \textbf{A\ref{assum:A0}},     \textbf{A\ref{assum:A1}},
\textbf{A\ref{assum:A2}}             and             \textbf{A\ref{assum:A4}},
Corollary~\ref{cor:Janson} applies and guarantees  that almost surely, for all
$\tau \in \llbracket t\rrbracket$,
\begin{align}
  \notag
  \Exp
  &\left[\left|\Zeta_{j,\tau}\right|^{p}  \middle|  \bar{Z}_{\tau},  F_{\tau-1}\right]
    \1\{\var_{j,\tau} \leq V\}\\
  \notag
  &\leq       \frac{3\pi}{2}
    \left[\left(\frac{15b_{2}\deg(\calG)}{2|\calA|}\right)^{p}     p^{p}     +
    \left(\frac{32V\deg(\calG)}{|\calA|}\right)^{p/2} 
    p^{p/2}\right] \1\{\var_{j,\tau} \leq V\}\\
  \label{eq:pth:moment}
  &\leq       \frac{3\pi}{2}
    \left[\left(\frac{15b_{2}\deg(\calG)}{2|\calA|}\right)^{p}     p^{p}     +
    \left(\frac{32V\deg(\calG)}{|\calA|}\right)^{p/2}     p^{p/2}\right].
\end{align}

It is now easy to show  that $\widetilde{\var}_{t,j} \leq v_{2}$ almost surely
(see    \eqref{eq:def:v2}    for    the   definition    of    $v_{2}$).     By
\textbf{A\ref{assum:A4}},  it  holds  that  $\var_{j,\tau}\leq  v_{1}$  almost
surely for each $\tau\in \llbracket t\rrbracket$, hence
\begin{equation*}
  \widetilde{\var}_{j,t} 
  =            \frac{1}{t}\sum_{\tau=1}^{t}
  \Exp\left[(\Zeta_{j,\tau})^{2}           \middle|          \bar{Z}_{\tau},
    F_{\tau-1}\right] = \frac{1}{t}\sum_{\tau=1}^{t}
  \Exp\left[(\Zeta_{j,\tau})^{2} \middle| \bar{Z}_{\tau},
    F_{\tau-1}\right] \1\{\var_{j,\tau}\leq v_{1}\} \leq v_{2}
\end{equation*}
because of \eqref{eq:pth:moment} with $p=2$.

We  now   turn  to  the   proof  of  \eqref{eq:thm:Bernstein}.   In   view  of
\eqref{eq:pth:moment}, by convexity of $u \mapsto |u|^{p}$, it holds that
\begin{align}
  \notag
  \Exp
  \left[\left|\widehat{H}_{j,t}       -       \widetilde{H}_{j,t}\right|^{p}
  \1\{\widetilde{\calF}_{V}\}\right] 
  &\leq                      \frac{1}{t}                      \sum_{\tau=1}^{t}
    \Exp\left[\left|\Zeta_{j,\tau}\right|^{p}\1\{\widetilde{\calF}_{V}\}\right]\\
  \notag
  &\leq                      \frac{1}{t}                      \sum_{\tau=1}^{t}
    \Exp\left[\left|\Zeta_{j,\tau}\right|^{p}\1\{\var_{j,\tau}
    \leq V\}\right]\\
  \notag
  &=                      \frac{1}{t}                      \sum_{\tau=1}^{t}
    \Exp\left[\Exp\left[\left|\Zeta_{j,\tau}\right|^{p}|\middle|
    \bar{Z}_{\tau}, F_{\tau-1}\right]\1\{\var_{j,\tau} 
    \leq V\}\right]\\
  \label{eq:convexity:argument}
  &\leq                                                         \frac{3\pi}{2}
    \left[\left(\frac{15b_{2}\deg(\calG)}{2|\calA|}\right)^{p} p^{p} + 
    \left(\frac{32V\deg(\calG)}{|\calA|}\right)^{p/2} p^{p/2}\right].
\end{align}
Therefore Markov's inequality implies that, for all $x > 0$,
\begin{align}
  \notag
  \Prob\left[\left|\widehat{H}_{j,t}    -   \widetilde{H}_{j,t}\right|    \geq
  x, \widetilde{\calF}_{V}\right]
  &\leq
    \Exp\left[x^{-p}\left| \widehat{H}_{j,t} -
    \widetilde{H}_{j,t}\right|^{p}\1\{\widetilde{\calF}_{V}\}\right]\\
  \label{eq:Markov:v1}
  &\leq \frac{3\pi}{2}\left(\frac{15b_{2}\deg(\calG)p/2
    + \sqrt{32|\calA|V\deg(\calG)p}}{x|\calA|}\right)^{p}.
\end{align}
By the technical Lemma~\ref{lem:tech}, there exists $p_{x} > 0$ such that
\begin{align*}
  x|\calA|
  &=    15eb_{2}\deg(\calG)p_{x}/2   +    \sqrt{32e^{2}|\calA|V\deg(\calG)p_{x}},
    \quad\text{and}\\
  p_{x} \geq q_{x}
  &:=  (x|\calA|)^{2} \left(32e^{2}|\calA|V\deg(\calG)  +
    15eb_{2}\deg(\calG)x|\calA|\right)^{-1}\\
  &=  x^{2}|\calA| \left(32e^{2}V\deg(\calG)  +
    15eb_{2}\deg(\calG)x\right)^{-1}.
\end{align*}
If   $q_{x}  \geq   2$,  then   $p_{x}$  is   a  valid   choice  for   $p$  in
\eqref{eq:Markov:v1}. This choice yields the inequality
\begin{equation*}
  \label{eq:Markov:one}
  \Prob\left[\left|\widehat{H}_{j,t} - \widetilde{H}_{j,t}\right| \geq
    x,  \widetilde{\calF}_{V}\right]  \leq  \frac{3\pi}{2}  \exp(-p_{x})  \leq
  \frac{3\pi}{2} \exp(-q_{x})   \leq \exp\left(2-q_{x}\right). 
\end{equation*}
Otherwise,
$\Prob[|\widehat{H}_{j,t}       -      \widetilde{H}_{j,t}|       \geq      x,
\widetilde{\calF}_{V}]  \leq \exp(2-q_{x})$  holds trivially.   This completes
the proof of Theorem~\ref{thm:Bernstein}.~\hfill$\square$

\begin{lemma}
  \label{lem:tech}
  For   any  $a,b,c>0$,   there  exists   $p>0$  such   that  $c=b\sqrt{p}   +
  ap$. Moreover, $c^{2}\leq (b^{2} + 2ac)p$. 
\end{lemma}

\begin{proof}[Proof of Lemma~\ref{lem:tech}]
  The quadratic  equation $c=bX +  aX^{2}$ has  a positive solution,  so there
  does   exist    $p>0$   such    that   $c=b\sqrt{p}   +    ap$.    Moreover,
  $c^{2}/p  =   b^{2}  +  2ab\sqrt{p}  +   a^{2}  p$  on  the   one  hand  and
  $2ac = 2ab\sqrt{p}  + 2a^{2}p \geq 2ab\sqrt{p} + a^{2}p$  on the other hand,
  implying that $c^{2}/p \leq b^{2} + 2ac$. This completes the proof.  
\end{proof}

\end{document}